\documentclass[12pt]{article}
\usepackage{amssymb}
\usepackage[mathscr]{eucal}
\usepackage{amsthm,amsmath,anysize}
\def\a{\alpha}
\def\b{\beta}
\def\l{\lambda}
\def\g{\gamma}
\def\0{\bar{0}}
\def\1{\bar{1}}
\def\e{\epsilon}
\def\d{\delta}
\def\g{\mathfrak{g}}
\def\m{\mathfrak m}

\def\g{{\mathfrak g}}

\newtheorem{lemma}{Lemma}[section]
\newtheorem{theorem}[lemma]{Theorem}

\newtheorem{proposition}[lemma]{Proposition}
\newtheorem{definition}[lemma]{Definition}
\newtheorem{corollary}[lemma]{Corollary}

\hoffset \voffset \oddsidemargin=60pt \evensidemargin=45pt
\title{ Representations of the restricted Lie superalgebra $p(n)$}

\author{Chaowen Zhang\\ Department of
Mathematics,\\ China University
 of Mining and Technology,\\ Xuzhou, 221116, Jiang Su, P. R. China}

\begin{document}
\maketitle

\par

\section{Introduction}

Let $F$ be an algebraically closed field of characteristic $p>2$, and let $\g$ be the restricted Lie superalgebra $p(n)$ over
$F$. In this paper,  we first study the center of a quotient of the restricted enveloping superalgebra of $\g$, then apply the obtained results  to investigate the restricted  representations of $\g$. \par

 For $n\geq 2$, the Lie superalgebra $\g$ over $F$ consists of block matrices of the form

  $$
  \begin{pmatrix}
  a& b\\
  c&-a^t
  \end{pmatrix}
  $$
  where $a$ is an arbitrary $n\times n$ matrix, $a^t$ denotes the transpose of $a$, $b$ is a symmetric $n\times n$ matrix and $c$ is a skew-symmetric $n\times n$ matrix (\cite{kk, kv, sv}). \par
   It is clear that $\g_{\bar 0}\cong \mathfrak{gl}(n)$ and hence $\g_{\0}'\cong \mathfrak{sl}(n)$ (\cite{kv, sv}).  The Lie superalgebra $\g$ has the $\mathbb{Z}$-gradings $\g=\g_{-1}\oplus \g_{\bar 0}\oplus \g_1$, where $\g_{-1}$ (resp. $\g_1$) consists  of above matrices  with $a=0$ and $b=0$ (resp. $a=0$ and $c=0$). \par

    According to \cite{kk, sv}, the universal enveloping superalgera of the Lie superalgebra $p(n)$ over $\mathbb C$ has a nontrivial radical.  Applying a similar idea, we define a nontrivial nilpotent ideal $I$ for the restricted enveloping superalgebra $u(\g)$. Then the study of  simple $u(\g)$-modules is reduced to  that of the  quotient superalgebra $\bar{\mathfrak u}=u(\g)/I$. \par
   The paper is organized as follows. Section 2 gives the preliminaries.  In Section 3 we define the quotient superalgebra   $\bar{\mathfrak u}$ and study its center.  Using the results from this section we study restricted representations  of $\g$ in Section 4.

   \section{Preliminaries}
    A Lie superalgebra $L=L_{\bar 0}\oplus L_{\bar 1}$ is called restricted if $L_{\bar 0}$ is a restricted Lie algebra and if $L_{\1}$ is a restricted $L_{\bar 0}$-module under the adjoint action.  Define the p-mapping $[p]$ on $\g_{\0}$ to be the $p$-th power for each matrix in $\g_{\0}$. Then $\g_{\0}$ becomes a restricted Lie algebra and   $\g$ is a restricted Lie superalgebra.\par
 Let $L=L_{\0}\oplus L_{\1}$ be a restricted Lie superalgebra, let $U(L)$ be its universal enveloping algebra and  let $M=M_{\0}\oplus M_{\bar 1}$ be a $L$-module. We say that $M$ admits a $p$-character $\chi\in L_{\0}^*$ if $(x^p-x^{[p]}-\chi^p(x)\cdot 1)M=0$ for all $x\in L_{\bar 0}$, where $x^p$ denotes the $p$-power of $x$ in $U(L)$.  According to \cite[Lemma 2.1]{zc1}, each simple $L$-module admits a unique $p$-character.\par
   Let $\chi\in\g_{\0}^*$ and let $I_{\chi}$ be the (super) ideal of $U(\g)$ generated by elements $x^p-x^{[p]}-\chi(x)^p$, $x\in\g_{\0}$. The superalgebra $u_{\chi}(\g)=U(\g)/I_{\chi}$ is called the $\chi$-reduced enveloping algebra of $\g$. If $\chi=0$, the superalgebra $u_{\chi}(\g)$,  referred to as the reduced enveloping algebra of $\mathfrak g$, is simply denoted $u(\g)$. The $\chi$-reduced enveloping algebra $u_{\chi}(\g_{\0})$ of the Lie algebra $\g_{\0}$ is defined similarly (see \cite[5.3]{sf}). In particular, $u_{\chi}(\g_{\0})$ may be viewed  as a subalgebra of $u_{\chi}(\g)$.\par

   Let $\mathfrak h$ be the  maximal torus of $\g_{\0}$ consisting of diagonal matrices. Then $\mathfrak h$ has a  basis $$h_1=:e_{11}-e_{n+1,n+1}, \dots, h_n:=e_{nn}-e_{2n,2n}.$$ Let $\epsilon_1,\dots, \epsilon_n\in \mathfrak h^*$ be the basis defined by $\epsilon_i(h_j)=\delta_{ij}$.
   Take the  natural basis of $\g$ as follows. For $1\leq i<j\leq n$,  let $y_{ij}=:e_{i+n, j}-e_{j+n,i}\in \g_{-1}$;  for  $1\leq i\leq j\leq n$, let $z_{ij}=e_{i,j+n}+e_{i+n,j}\in \g_1$; for  $1\leq i,j\leq n$, let $\tilde e_{ij}=e_{ij}-e_{j+n,i+n}\in \g_{\0}$. These are root vectors of $\g$  and  the roots are of the following form.

$$
\begin{aligned}
\text{Roots\  of}\ & \g_{\0}:&& \e_i-\e_j,&&&  1\leq i\neq j\leq n;\\
\text{Roots\  of}\ &  \g_{-1}:&& -\e_i-\e_j,&&& 1\leq i<j\leq n;\\
\text{Roots\  of}\ &  \g_1:&& \e_i+\e_j,&&& 1\leq i\leq j\leq n.
\end{aligned}
$$
Note that $\text{dim}\g_{-1}=d=:{1\over 2}n(n-1)$ and $\text{dim}\g_1=\frac{1}{2}n(n+1)$.\par

 Denote by $G$ the  linear algebraic group $GL_n(F)$. Let $$\rho:\  GL_n(F)\longrightarrow GL_{2n}(F)$$ be defined  by $\rho (a)=\text{diag} (a, a^{-1,t})$. Then $\rho$ is a monomorphism of algebraic groups (\cite[7.4]{hu}). Under the action  $$ Ad\rho(a) x=\rho(a)x\rho (a)^{-1},\  a\in G, x\in \g,$$ the Lie superalgebra $\g$ becomes a (rational) $G$-module (see \cite[Corollary 10.3]{hu}).  Then  $U(\g)$ is also a $G$-module.

 \par
Let $g\in G$ and let $$x=\begin{pmatrix}0&b\\
                               0& 0 \end{pmatrix}\in \g_1.$$ Then we get $$Ad\rho(g) x=\begin{pmatrix}0&gbg^t\\
                               0& 0 \end{pmatrix}\in \g_1.$$
  Hence $\g_1$ is a $G$-submodule of $\g$ by \cite[2.9(2), Part I]{jj1}. Similarly $\g_{-1}$ is also a $G$-submodule of $\g$.\par

Let $\g^+=\g_{\0}\oplus \g_1$. According to \cite[Lemma 2.4]{zc1}, each simple $U(\g^+)$-module is annihilated by $\g_1$. Let $\chi\in \g^*_{\0}$, and let $M$ be a simple $u_{\chi}(\g_{\0})$-module viewed as a $u_{\chi}(\g^+)$-module by letting $\g_1\cdot M=0$. Define the Kac module $$K_{\chi}(M)=u_{\chi}(\g)\otimes _{u_{\chi}(\g^+)}M.$$
Since $M$ is finite dimensional (\cite[Theorem 5.24]{sf}),  so is $K_{\chi}(M)$. If $\chi=0$, then we denote $K_{\chi}(M)$  simply by $K(M)$.\par
   We have $\g_{-1}$-module isomorphism $K_{\chi}(M)\cong \wedge (\g_{-1})\otimes_F M$, where $\wedge (\g_{-1})$ is the exterior algebra of $\g_{-1}$. Let $v_1,\dots, v_d$ be a basis of $\g_{-1}$. Set $$\wedge^k(\g_{-1})=\langle v_{i_1}\cdots v_{i_k}\in \wedge (\g_{-1})|1\leq i_1\leq \cdots\leq i_k\leq d\rangle.$$ Denote $\wedge_{\0}(\g_{-1})=\sum_{k=2i}\wedge^k(\g_{-1})$ and $\wedge_{\1}(\g_{-1})=\sum_{k=2i+1}\wedge^k(\g_{-1})$. Then $\wedge (\g_{-1})$ is $\mathbb Z_2$-graded with $\wedge(\g_{-1})=\wedge_{\0}(\g_{-1})\oplus \wedge_{\1}(\g_{-1})$. It follows that $K_{\chi}(M)$ is $\mathbb Z_2$-graded with $$K_{\chi}(M)_{\bar i}=\wedge_{\bar i} (\g_{-1})\otimes_{F} M, \ i=0,1.$$

\begin{lemma}Let $\chi\in \g_{\0}^*$ and let $\m=\m_{\0}\oplus \m_{\1}$ be a simple $u_{\chi}(\g)$-module. Then there is an (even) epimorphism $K_{\chi}(M)\longrightarrow \m$ with $M\subseteq \m_{\0}$  a simple $u_{\chi}(\g_{\0})$-submodule.
\end{lemma}
\begin{proof}
Let $N\subseteq \m_{\0}$ be a simple $u_{\chi}(\g_{\0})$-submodule. Then  $\wedge^k (\g_1) N\subseteq \m_{\0}$ if $k$ is even, and $\wedge^k (\g_1)N\subseteq \m_{\1}$ if $k$ is odd. It is clear that each $\wedge^k (\g_1) N$ is a $u_{\chi}(\g_{\0})$-submodule of $\m$. Let $k$ be the largest such that $\wedge^k (\g_1) N\neq 0$, and let $M\subseteq \wedge^k (\g_1)N$ be a simple $u_{\chi}(\g_{\0})$-submodule. Then we have $\g_1 M=0$. We may renumber the $\mathbb Z_2$-gradings of $\m$ if necessary, so we always assume $M\subseteq \m_{\0}$. Then since $\m=\m_{\0}\oplus \m_{\1}$ is simple, we obtain an (even) epimorphism from $K_{\chi}(M)$ to $\m$.
\end{proof}
Let $\mathfrak n_0^+$ (resp. $\mathfrak n_0^-$)  be the Lie subalgebra of $\g_{\0}$ spanned by positive root vectors $$\tilde e_{ij}=e_{ij}-e_{j+n,i+n}, \  i<j\ (\text{resp.} \ i>j).$$
\begin{definition} Let $\m=\m_{\0}\oplus \m_{\1}$ be a $\g$-module.   A nonzero (homogeneous) vector $v\in \m$ is  {\sl maximal}  of weight $\mu=\mu_1\e_1+\cdots +\mu_n\e_n\in \mathfrak h^*$ if  $(\mathfrak n_0^++\g_1)\cdot v=0$ and if $h_i\cdot v=\mu_iv$, $i=1,\dots, n$.
\end{definition}

 If $\chi (\mathfrak n_0^+)=0$, then each simple $u_{\chi}(\g)$-module  has  a maximal vector (see \cite[Section 2]{zc1}).  For $\mu=\sum^n_{i=1}\mu_i\e_i\in \mathfrak h^*$, set $\delta_{\mu}=:\Pi_{i<j}(\mu_i-\mu_j+j-i-1)\in F$.
 \par
 In the following proposition, we assume  $\chi(\mathfrak n_0^+)=0$ and let $M$ be a simple $u_{\chi}(\g_{\0})$-module generated by a maximal vector of weight $\mu\in\mathfrak h^*$. \par

\begin{proposition} If $\delta_{\mu}\neq 0$, then  $K_{\chi}(M)$ is simple.
\end{proposition}
\begin{proof} For an arbitrarily fixed order of the basis vectors $y_{ij}$'s  of $\g_{-1}$, set $Y=\Pi_{i<j}y_{ij}\in \wedge^d (\g_{-1})$.\par
 Let $\m\subseteq K_{\chi}(M)$ be a simple $u_{\chi}(\g)$-submodule.
 We claim that $Y\otimes M\subseteq \m$. To see this, take a nonzero vector $$x=\sum_{i_1,\dots, i_k}  y_{i_1,\dots, i_k}\otimes m_{i_1,\dots, i_k}\in \m,$$ where each $y_{i_1,\dots, i_k}$ is a monomial in $\wedge^k(\g_{-1})$  and $m_{i_1,\dots, i_k}\in M$. By applying appropriate $y_{ij}$'s to $x$, we obtain  $Y\otimes m\in \m$ for some nonzero element $m\in M$. Since $FY$ is a $1$-dimensional $\text{ad}_{\g_{\0}}$-module,  and since $M$ is a simple $u_{\chi}(\g_{\0})$-module, we obtain $Y\otimes M\subseteq \m$.\par
     Set $Z=\Pi_{i<j}z_{ij}$, where the $z_{ij}$'s are in an arbitrary fixed order. Let $v_{\mu}\in M$ be the maximal vector of the weight $\mu$. Applying a similar argument as that used in the  proof of \cite[Lemma 3.1]{kv}, we have $$ZY\otimes v_{\mu}=\pm \delta_\mu \otimes v_{\mu}\in \mathfrak m.$$ Since $\delta_{\mu}\neq 0$, we have $1\otimes v_{\mu}\in \m$, and hence $\m=K_{\chi}(M)$,  implying  that $K_{\chi}(M)$ is simple.
\end{proof}

\section{The quotient algebra $\bar {\mathfrak u}$}
\subsection{ A nilpotent (super) ideal of $u(\g)$}
Since $\g$ is a restricted Lie superalgebra,  the  element $x^p-x^{[p]}$ is central in $U(\g)$ for every $x\in \g_{\0}$.
By \cite[Theorem 5.1.2]{sf}, these elements generates a polynomial algebra $\mathcal O$ inside $U(\g_{\0})\subseteq U(\g)$. Let $x\in\g_{\0}$. Since $x^{[p]}$ is the $p$-th power of the matrix $x$, we have $gx^{[p]}g^{-1}=(gxg^{-1})^{[p]}$ for every $g\in G$. Therefore,  $\mathcal O$ is $G$-invariant. Then the (super) ideal $I_0$ defined in Section 2 is a $G$-submodule of $U(\g)$. It follows that $u(\g)$ is a $G$-module.\par

In the following we assume $p>3$. Let $x_1,\dots, x_t$ be a basis of $\g_{\0}$, $y_1,\dots, y_d$ be a basis of $\g_{-1}$ and $z_1, \dots, z_l$ be a basis of $\g_1$.  By \cite[Theorem 2.5]{bmpz} $u(\g)$ has a basis of the form
$$y^{\d}x^sz^{\d'}=:y_1^{\d_1}\cdots y_d^{\d_d}x_1^{s_1}\cdots x_t^{s_t}z_1^{\d'_1}\cdots z_l^{\d'_l},\  \d_i,\d'_i=0,1,\ 0\leq s_j\leq p-1.$$
 It follows that $$u(\g)\cong \wedge (\g_{-1})\otimes u(\g_{\0})\otimes \wedge (\g_1).$$
\par
 If $y=y_{i_1}\cdots y_{i_q}\in \wedge^q (\g_{-1})$ or $z=z_{i_1}\cdots z_{i_q}\in\wedge^q (\g_1)$, we call $q$ the length of $y$ or $z$, and write $l(y)=q$ or $l(z)=q$.
The following lemma is proved in \cite{kk} for the Lie superalgebra $p(n)$ over  $\mathbb C$. Applying almost verbatim the proof there we obtain the following conclusion in $u(\g)$.
\begin{lemma}(cf. \cite[Lemma 2.3(b)]{kk}) Let $y\in \wedge(\g_{-1})$ and $z\in \wedge(\g_1)$ are two monomials. Then $zy=\sum_{\a} u_{\a}y_{\a}z_{\a}$, where $u_{\a}\in u(\g_{\0})$, $y_{\a}\in \wedge (\g_{-1})$ and $z_{\a}\in\wedge (\g_1)$ such that $l(y_{\a})-l(z_{\a})=l(y)-l(z)$.
\end{lemma}
 Let $u(\g)_k$ be spanned by above basis elements of $u(\g)$ with $\sum_j{\d_j'}-\sum_i{\d_i}=k$. Then we get $$u(\g)=\oplus _{k=-d}^{l}u(\g)_k.$$ By Lemma 3.1,  we obtain  $u(\g)_iu(\g)_j\subseteq u(\g)_{i+j}$. So that $u(\g)$ is a $\mathbb Z$-grade algebra.
 It is easy to see that each $u(\g)_k$ is a  $G$-submodule.\par

Let $I$ denote the (super) ideal of $u(\g)$ generated by $\wedge^k (\g_1)$ with $k>d$. Then $I$ is spanned by  basis elements  $y^{\d}x^sz^{\d'}$  with $\sum \d'_j\geq d+1$. It follows that $I$ is a $G$-submodule of $u(\g)$ such that $$I\subseteq \sum_{k\geq 1}u(\g)_k.$$ Applying Lemma 3.1 we have $I^s\subseteq \sum_{k\geq s}u(\g)_k$ for any $s\geq 1$, implying that $I$ is nilpotent.\par

We say that $\mu\in\mathfrak h^*$ is compatible with $\chi\in\g_{\0}^*$ if $\mu (h)^p-\mu (h)=\chi (h)^p$ for all $h\in \mathfrak h$.
 Set $\mathfrak b_0=:\mathfrak h+\mathfrak n_0^+\subseteq \g_{\0}$. \par
 Let $\chi\in\g_{\0}^*$ be such that $\chi(\mathfrak n^+_0)=0$ and let $\mu$ be compatible with $\chi$.  Let $Fv_{\mu}$ be the 1-dimensional $\mathfrak b_0 +\g_1$-module of weight $\mu\in \mathfrak h^*$ and annihilated by $\mathfrak n_0^++\g_1$. Define the baby Verma (super) module $$Z_{\chi}(\mu)=u_{\chi}(\g)\otimes _{u_{\chi}(\mathfrak b_0 +\g_1)}  Fv_{\mu}.$$ If $\chi=0$, we denote it simply by $Z(\mu)$.\par

The following lemma follows immediately from Lemma 3.1.
\begin{lemma} (1) Let $K_{\chi}(M)$ be a Kac module. Then  $I  K_{\chi}(M)=0$.\par
 (2) Let $Z_{\chi}(\mu)$ be a baby Verma module. Then  $I Z_{\chi}(\mu)=0$.
\end{lemma}

\subsection{The definition of  $\bar {\mathfrak u}$}
Define the quotient superalgebra $\bar{\mathfrak u}=u(\g)/I$. From Lemma 3.2  we see that each Kac module $K(M)$, each baby Verma module $Z(\mu)$ or each simple module of $u(\g)$ may be viewed as a $\bar{\mathfrak u}$-module.\par

From above discussions we have $\bar{\mathfrak u}=\oplus_{k=-d}^d \bar{\mathfrak u}_k$, where  each $\bar{\mathfrak u}_k$ has a basis  $$y^{\d}x^s z^{\d'},\ \sum_j \d'_j-\sum_i\d_i=k.$$  Set $$\mathfrak a=\langle y^{\d}x^s z^{\d'}|\ \sum_j \d'_j=\sum_i\d_i>0\rangle.$$ Then we have $\bar{\mathfrak u}_0=u(\g_{\0})\oplus \mathfrak a$. It is clear that $\mathfrak a$, $u(\g_{\0})$ and each $\bar{\mathfrak u}_k$ are all $G$-submodules of $\bar{\mathfrak u}$.

\subsection{The center of $\bar{\mathfrak u}$}
Recall the imbedding $\rho$ of $G$ into $GL_{2n}(F)$. We have $$\rho(G)=\{\text{diag}(a, a^{-1,t})|a\in GL_n( F)\}.$$
The differential $\bar\rho$: $\mathfrak{gl}(n,F)\longrightarrow \mathfrak{gl}(2n, F)$ is a Lie algebra homomorphism. Using \cite[5.4]{hu} we obtain $$\bar\rho (a)=\begin{pmatrix}a&0\\
                               0& -a^t \end{pmatrix}\in \g_{\0}, \ \  a\in \mathfrak{gl} (n,F).$$

 Define a bilinear form $\theta: \g_{\0}\times \g_{\0}\longrightarrow  F$ by letting $\theta (x, y)=\text{tr} (ab)$ $$\text{for}\ x=\text{diag} (a, -a^t),\  y=\text{diag} (b, -b^t)\in \g_{\0}.$$ Then it is easy to see that $\theta$ is symmetric, non-degenerate and $G$-invariant. Besides, the restriction of $\theta$ to $\mathfrak h$ is also non-degenerate. For each $\l\in\mathfrak h^*$, let $h_{\l}\in\mathfrak h$ be the unique element such that $\theta (h_{\l}, h)=\l (h)$ for all $h\in\mathfrak h$. It is easy to show that $h_{\e_i-\e_j}=h_i-h_j$ for any $i\neq j$.\par

  Note that $G$ satisfies the assumptions (H1)-(H3) in \cite[6.3]{jj}.
Recall the adjoint action  $Ad\rho$ of $G$ on the Lie superalgebra $\g$.
\begin{lemma} The differential of  $Ad\rho$ is $ad\bar\rho: \mathfrak{gl}(n, F)\longrightarrow \mathfrak{gl}(\g)$ given by $$ad\bar \rho (x)(y)=[\bar\rho (x), y]=\bar \rho (x) y-y\bar\rho (x),\ \ x\in \mathfrak{gl}(n, F),\  y\in \g.$$
\end{lemma}
\begin{proof} Since $d(Ad \rho)=d(Ad)\bar\rho$ by \cite[5.4]{hu}, the lemma follows from \cite[Theorem 10.4]{hu}.
\end{proof}

 Define $$\mathcal Z=\{u\in \bar {\mathfrak u}_{\0}|\ gu=u\ \text{for all}\  g\in G\  \text{and}\ [u, \g_1]=0\}.$$
Since $ \bar{\mathfrak u}$ is a $G$-module, we obtain from \cite[7.11(5),I]{jj1} that $ \mathcal Z\subseteq \{u\in\bar{\mathfrak u}_{\bar 0}| [x,u]=0 \ \text{for all}\ x\in \g\}$.
Thus, $\mathcal Z$ is a central subalgebra of  of $\bar{\mathfrak u}$. \par
 Take the maximal torus $T=\{diag(t_1,\cdots t_n)|t_i\in F^*\}$ of $G$. We also use the notation $\e_i$, $1\leq i\leq n$, to denote the natural basis of $X(T)$. Therefore the notation $\e_i\in \mathfrak h^*$ is the differential of $\e_i\in X(T)$ (see \cite[1.2]{jj4}). Recall the root vectors of $\g$ in Section 2. Then the $T$-weights of these root vectors (viewing $X(T)$ as a additive group)  are $$\text{wt} (y_{ij})=\e_i+\e_j,\  \text{wt}(z_{ij})=-(\e_i+\e_j),\  \text{wt}(\tilde e_{ij})=\e_i-\e_j.$$
  Set $\bar{\mathfrak u}^T=\{x\in \bar{\mathfrak u}|tx=x \ \text{for all}\  t\in T\}$.
 \begin{lemma} $\bar{\mathfrak u}^T\subseteq \bar{\mathfrak u}_0$.
 \end{lemma}
 \begin{proof} Recall that $\bar {\mathfrak u}$ has a basis consisting of elements $y^{\d}x^sz^{\d'}$.  Let $u\in  \bar{\mathfrak u}^T$.  Since  each of these basis elements is a $T$-weight vector, it suffices to assume $u$ is such an element. Write  $$u=\Pi_{i<j}y_{ij}^{\d_{ij}}\Pi_{i\neq j}\tilde e^{s_{ij}}_{ij}f(h)\Pi_{i\leq j}z_{ij}^{\d'_{ij}},\  \d_{ij}, \d'_{ij}=0,1,\  0\leq s_{ij}\leq p-1, \ f(h)\in u(\mathfrak h).$$ Note that $\text{wt}(f(h))=0$.\par
 Since $tu=t^{\text{wt}(u)}u=u$ for all $t\in T$, we have $\text{wt}(u)=0$. Also, we have $$\begin{aligned}
  \text{wt}(\Pi_{i\neq j}\tilde{e}_{ij}^{s_{ij}})&=\sum_{i\neq j}s_{ij}(\e_i-\e_j)\\ &=\sum_{i<j}l_{ij}(\e_i-\e_j),\end{aligned}$$ where $l_{ij}=s_{ij}-s_{ji}$.
 It follows that  $$\sum_{i<j}(\d'_{ij}-\d_{ij})(\e_i+\e_j)+\sum_{i=1}^n2\d'_{ii}\e_i+\sum_{i<j}l_{ij}(\e_i-\e_j)=0.$$
 Since $\e_1,\dots,\e_n\in X(T)$ are linearly independent, the coefficient of each $\e_i$ is $0$. Then we have $$\sum_{i<j}(\d'_{ij}-\d_{ij})+\sum_{k<i}(\d'_{ki}-\d_{ki})+2\d'_{ii}+\sum_{i<j}l_{ij}-\sum_{k<i}l_{ki}=0.$$
 The sum of all these $n$ equations gives $$2\sum_{i<j}\d'_{ij}-2\sum_{i<j}\d_{ij}+2\sum_{i=1}^n\d'_{ii}=0,$$
 and hence $\sum_{i\leq j}\d'_{ij}=\sum_{i<j}\d_{ij}$, implying that $u\in \bar{\mathfrak u}_0$.
 \end{proof}
 Immediately from the lemma, we have $\mathcal Z\subseteq \bar{\mathfrak u}_0$.
 \subsection{The Harish-Chandra homomorphism on  $\bar{\mathfrak u}$}

 Let $\Lambda_0=F_p\e_1+\cdots +F_p\e_n$ and  let  $u(\mathfrak h)$ be the restricted enveloping algebra of $\mathfrak h$ (see \cite[Section 5.3]{sf}).
 \begin{lemma} Let $f(h)\in u(\mathfrak h)$. If  $f(h)(\mu)=0$ for all $\mu\in \Lambda_0$, then $f(h)=0$.
 \end{lemma}
\begin{proof} We proceed by induction on $n$. If $n=1$, then $f(h)$ is a polynomial of one variable $h_1$ such that $\text{deg} f\leq p-1$. Since the set $\{h_1(\mu)|\ \mu\in\Lambda_0\}\subseteq F$ contains $p$ elements,  we obtain $f(h)=0$.\par If $n\geq 2$, then we may write $f(h)\in u(\mathfrak h)$ as $$\begin{aligned}
&f(h_1,\dots, h_n)\\ &=f_0(h_1,\dots,h_{n-1})+f_1(h_1,\dots,h_{n-1})h_n+\cdots +f_{p-1}(h_1,\dots,h_{n-1})h_n^{p-1}.\end{aligned}$$ For each fixed
$\mu\in\Lambda_0$, using the fact that $\mu+F_p\e_n\subseteq \Lambda_0$ together with the conclusion for $n=1$, we obtain $$\begin{aligned} f_0(h_1(\mu),\dots,h_{n-1}(\mu)) &=f_1(h_1(\mu),\dots,h_{n-1}(\mu))\\
&=\cdots \\
&=f_{p-1}(h_1(\mu),\dots, h_{n-1}(\mu))\\
&=0.\end{aligned}$$ Then the induction hypotheses yields $f_i(h_1,\dots, h_{p-1})=0$ for all $i$, and hence $f(h)=0$.
\end{proof}
Denote by $\mathfrak n^+$(resp. $\mathfrak n^-$) the Lie subalgebra of $\g$ spanned by positive (resp. negative) roots.  We fix an order $$f_{\a_1},\dots, f_{\a_l}, h_1,\dots, h_n, e_{\b_1}\dots, e_{\b_m},$$ where $\a_1,\dots, \a_l$ (resp. $\b_1,\dots \b_m$) are all the positive (resp. negative) roots of $\g$. From Section 3.2,  $\bar {\mathfrak u}$ has a basis of the form  $$f_{\a_1}^{d_1}\cdots f_{\a_l}^{d_l}h_1^{s_1}\cdots h_n^{s_n}e_{\b_1}^{d_1'}\cdots e_{\b_m}^{d_m'},\ 0\leq s_k\leq p-1, \sum_{\beta_i\ \text{is odd}}d_i'\leq d,$$ where  $d_i=0,1$ (resp. $d_i'=0,1$) if   $\a_i$ (resp. $\beta_i$) is odd and $0\leq d_i\leq p-1$ (resp. $0\leq d_i'\leq p-1$) if   $\a_i$ (resp. $\beta_i$) is even.  Then $\bar{\mathfrak u}^T$ is spanned by these vectors with $$\sum_{i=1}^l d_i\a_i=\sum_{j=1}^m d'_j\beta_j.$$ Denote by $\mathfrak n^-\bar{\mathfrak u}$ (resp. $\bar{\mathfrak u}\mathfrak n^+$) the set spanned by these vectors with $\sum _id_i>0$ (resp. $\sum_i d_i'>0$).\par
The following conclusion is analogous to  \cite[Lemma 7.4.2]{jd},
 \begin{lemma}  Set $\mathfrak l=:\bar{\mathfrak u}\mathfrak n^+\cap \bar {\mathfrak u}^T$. \par
 (1) $\mathfrak l=\mathfrak n^- \bar{\mathfrak u}\cap \bar{\mathfrak u}^T$, so that $\mathfrak l\triangleleft  \bar{\mathfrak u}^T$;\par
 (2) $ \bar{\mathfrak u}^T= u(\mathfrak h)\oplus \mathfrak l$.
 \end{lemma}
 \begin{proof} It suffices to prove (2).\par It is clear that $\bar{\mathfrak u}^T=u(\mathfrak h)+ \mathfrak l$. Let $x\in u(\mathfrak h)\cap \mathfrak l$. Then we have $ x= u_0(h)$ for some  $u_0(h)\in u(\mathfrak h)$ and  $x\in\bar{\mathfrak u}\mathfrak n^+$. Applying $x$ to the maximal vector $v_{\mu}\in Z(\mu)$ for all $\mu\in \Lambda_0$.   Then  we get $u_0(h) v_{\mu}=0$, and hence $u_0(h)(\mu)=0$ for all $\mu\in \Lambda_0$. It follows that $x=u_0(h)=0$,  which completes the proof.
 \end{proof}
   Define the Harish-Chandra homomorphism  $h: \bar{\mathfrak u}^T\longrightarrow u(\mathfrak h)$ to be the projection with kernel $\mathfrak l$. Recall that $\bar{\mathfrak u}_0=u(\g_{\0})\oplus \mathfrak a$. It is easy to see that $\bar{\mathfrak u}^T=\bar{\mathfrak u}_0^T=u(\g_{\0})^T\oplus {\mathfrak a}^T$. Then the restriction of $h$ to $u(\g_{\0})^T$ and $u(\g'_{\0})^T$ define analogous homomorphisms on $u(\g_{\0})$ and $u(\g'_{\0})$.\par

Let $$\Phi_0=\{\e_i-\e_j| 1\leq i\neq  j\leq n\}\ (\text{resp.} \
\Phi^+_0=\{\e_i-\e_j| 1\leq i< j\leq n\}) $$ denote the root system (resp. positive roots) of $\g_{\0}$.
We denote the Weyl group of $G$ with respect to  $T$ by $W$. This group is generated by $s_{\a}, \a\in \Phi^+_0\subseteq X(T)$. The action of each $s_{\a}$ on $\mathfrak h^*$ is given by $s_{\a}(\l)=\l-\l(h_{\a})\a$. Since $p>2$, each $s_{\a}$ has order $2$. Since the elements  $$h_{\e_1-\e_2}=h_1-h_2, \dots, h_{\e_{n-1}-\e_n}=h_{n-1}-h_n$$ are linearly independent, we can find $\rho \in\mathfrak h^*$ such that $$\rho (h_{\e_i-\e_{i+1}})=1\  \text{for}\  1\leq i\leq n-1.$$ The dot action on $\mathfrak h^*$ of any $w\in W$ is defined by $w\cdot \l=w(\l+\rho)-\rho$. In particular, we have $s_{\a}\cdot \l=s_{\a}\l-\a$ for all simple $\a$. Since the $s_{\a}$ with $\a$ simple generate $W$, the dot action is independent of the choice of $\rho$ (see \cite[9.2]{jj}).\par

 Since the bilinear form $\theta_{|\mathfrak h}$ is $W$-invariant and since $W$ permutes $\Phi_0$ (\cite[Lemma 10.4 C]{hu1}, we obtain that $W$ permutes $\{h_{\a}| \a\in \Phi_0\}$.\par

We may identify $U(\mathfrak h)$ with the symmetric algebra $S(\mathfrak h)$, and hence with the algebra of  polynomial functions on $\mathfrak h^*$. Thus, the dot action on $\mathfrak h^*$ yields also a dot action on $U(\mathfrak h)$ defined by $(w\cdot f) (\l)=f(w^{-1}\cdot \l)$ for $f\in U(\mathfrak h)$. In particular, we have $s_{\a}\cdot h=s_{\a}(h)-\a (h) 1$ if $\a$ is a simple root (\cite[9.2]{jj}).\par
It is easy to see that $W$ stabilizes $\Lambda_0\subseteq \mathfrak h^*$. Then  $W$ also acts on $\Lambda_0$ by the dot action.  We have  $s_{\a}(h)=h-\a(h)h_{\a}$ for all simple roots $\a$, implying that $$s_{\a}(h_i)^p-s_{\a}(h_i)\in \mathcal O\cap U(\mathfrak h)$$ for all $i$. Then $W$ acts also on the restricted enveloping algebra $u(\mathfrak h)$.  Since $$(s_{\a}\cdot h_i)^p-s_{\a}\cdot h_i=s_{\a}(h_i)^p-s_{\a}(h_i)\in \mathcal O$$ for all $i$, $W$ also acts  on $u(\mathfrak h)$ by the dot action. \par

Following \cite{sv}, define an element $\Theta\in U(\mathfrak h)$ by  $$\begin{aligned} \Theta (\mu)& =\Pi_{\a\in\Phi_0} ((\a, \mu+\rho)-1)\\
&=\Pi_{\i\neq j} ((h_i-h_j)(\mu +\rho)-1)\  \text{for}\ \mu\in\mathfrak h^*.\end{aligned}$$
 Since $W$ permute $\Phi_0$, we have  $w\cdot \Theta =\Theta$ for all $w\in W$,  and hence $\Theta\in U(\mathfrak h)^{W\cdot}$. \par
  Note that $$\text{deg}\Theta=|\Phi_0|=2|\Phi_0^+|=2d.$$ \par
Assume $p>2d$ in the following. Denote also by $\Theta$ its image in $u(\mathfrak h)$. Then we have $\Theta\in u(\mathfrak h)^{W\cdot}$. \par
Let $\a\in\Phi^+_0$ and let $x_{\a}$ be a root vector such that $[x_{\a}, x_{-\a}]=h_{\a}$. We may simply let $x_{\a}={\tilde e}_{ij}$ and $x_{-\a}={\tilde e}_{ji}$ for $\a=\e_i-\e_j\in \Phi^+_0$.  Note that $ad^3 x_{\a}(\g)=0$. With the assumption $p>3$, we may define an automorphism of $\g$ by letting $$\tilde s_{\a}=(\text{exp}\ ad x_{\a})(\text{exp}\ ad x_{-\a})(\text{exp} \ ad x_{\a}).$$
By \cite[1.10.19]{jd}, we have ${\tilde s_{\a}}|_{\mathfrak h}=s_{\a}$. Thus, each element of $W$ may be viewed as an automorphism of $\g$. By essentially a similar argument as that for \cite[2.3 (*)]{hu1}, we obtain $g\in G$ such that $$\begin{aligned} \tilde s_{\a}(x^{[p]})&= \rho(g) x^{[p]}\rho(g)^{-1}\\ &
=(\rho(g)x\rho(g)^{-1})^{[p]}\\ &=(\tilde s_{\a}x)^{[p]}\ \text{for all}\  x\in \g_{\0}.\end{aligned}$$ Then $\tilde s_{\a}$ can be extended to an automorphism of $u(\g)$. It is easy to see that $\tilde s_{\a}(\g_{\pm 1})\subseteq \g_{\pm 1}$. Then each $w\in W$ can also be extended to an automorphism of $\bar{\mathfrak u}$.

\begin{lemma} $$h(\mathcal Z)\subseteq u(\mathfrak h)^{W\cdot}.$$
\end{lemma}
\begin{proof}  Let $\mu\in \Lambda_0\subseteq \mathfrak h^*$, and let $v_{\mu}$ be the maximal vector in $Z(\mu)$. The Lie superalgebra $\g$ has a triangular decomposition  $$\g=(\g_{-1}+\mathfrak n^-)+\mathfrak h +(\mathfrak n^+ +\g_1).$$  Let $x_{\a}$ be a root vector for $\a\in\Phi^+_0$.  Note that  $x_{-\a}^{p-1}\otimes v_{\mu}\in Z(\mu)$ is a maximal vector for the new triangular decomposition $$\begin{aligned}\g=& \tilde s_{\a}(\g_{-1}+\mathfrak n^-)+\mathfrak h +\tilde s_{\a}(\mathfrak n^+ +\g_1)\\  =& (\g_{-1}+\tilde s_{\a}(\mathfrak n^-))+ \mathfrak h +(\tilde s_{\a}(\mathfrak n^+)+\g_1).\end{aligned}$$  Therefore, $Z(\mu)$ is isomorphic to the baby Verma module $Z(\mu-(p-1)\a)$ with respect to the new triangular decomposition.\par

Let $u\in \mathcal Z$. By Lemma 3.6, we may write   $u=u_1+h(u)$ such that $u_1\in\bar{\mathfrak u}\mathfrak n^+$.  Since $\mathcal Z\subseteq \bar{\mathfrak u}^G,$ using the discussion preceding the lemma  we have $$u=\tilde s_{\a}(u)= \tilde s_{\a}(u_1) +s_{\a}(h(u)).$$ Since $\tilde s_{\a}(u_1)\in \bar{\mathfrak u}\tilde s_{\a}(\mathfrak n^+)$, the righthand side is a linear combination of a  basis of $\bar{\mathfrak u}$ using the new triangular decomposition.\par
We now use Jantzen's argument (see \cite[9.5]{jj}). Applying $u$ to $x_{-\a}^{p-1}\otimes v_{\mu}$, we have
$$u(x_{-\a}^{p-1}\otimes v_{\mu})=h(u)(\mu) x_{-\a}^{p-1}\otimes v_{\mu}$$ and $$u(x_{-\a}^{p-1}\otimes v_{\mu})=\tilde s_{\a}(u)(x_{-\a}^{p-1}\otimes v_{\mu})$$$$=s_{\a}(h(u))(\mu +\a)(x_{-\a}^{p-1}\otimes v_{\mu})=
h(u)(s_{\a}\cdot \mu)(x_{-\a}^{p-1}\otimes v_{\mu}),$$ implying that $s_{\a}\cdot h(u)=h(u)$, and hence  $h(u)\in u(\mathfrak h)^{W\cdot}$, since $\a$ is arbitrary. This completes the proof.
\end{proof}

\begin{lemma} Let $\mu\in\Lambda_0$ be such that $\mu_{n-1}=\mu_n$ and let $h(z)\in h(\mathcal Z)$. Then $$h(z)(\mu+t(\e_{n-1}+\e_n))=h(z)(\mu)$$ for any $t\in  F_p$.
\end{lemma}
\begin{proof}
Let $v_{\mu}\in Z(\mu)$ be the maximal vector of weight $\mu$. Recall the root vector $\tilde e_{n.n-1}=e_{n,n-1}-e_{2n-1,2n}\in\g_{\0}$.  Since $\mu_{n-1}=\mu_n$,
 $\tilde e_{n,n-1}v_{\mu}$ is a maximal vector of weight $\mu-(\e_{n-1}-\e_n)$. Then we obtain a homomorphism of $\bar{\mathfrak u}$-modules $\varphi: Z(\mu-(\e_{n-1}-\e_n))\longrightarrow Z(\mu)$. Since $v_{\mu}\notin \text{Im} \varphi$,  we have $\text{Im}\varphi\subsetneq Z(\mu)$.\par  Recall the root vector for $-\e_n-\e_{n-1}$ is  $y_{n-1,n}=e_{2n-1,n}-e_{2n,n-1}\in \g_{-1}$. A straightforward computation shows that, the image of the vector $y_{n-1,n}v_{\mu}$ in $Z(\mu)/\text{Im}\varphi$ is a maximal vector of weight $\mu-(\e_{n-1}+\e_n)$. This implies that $$h(z)(\mu)=h(z)(\mu-(\e_{n-1}+\e_n))$$ and hence $h(z)(\mu)=h(z)(\mu+t(\e_{n-1}+\e_n))$ for all $t\in  F_p$.
\end{proof}

\begin{lemma} Let $\mu\in\Lambda_0$ be such that $(\mu+\rho, \e_i-\e_j)=1$ for $i\neq j$.  Let $h(z)\in h(\mathcal Z)$. Then $h(z)(\mu+t(\e_i+\e_j))=h(z)(\mu)$ for any $t\in  F_p$.
\end{lemma}
\begin{proof} Let $\a=\e_s-\e_t, \ s<t$. Then $$s_{\a}(\e_k)=\begin{cases} \e_k, &\text{if $k\neq s,t$}\\ \e_s, &\text{if $k=t$}\\ \e_t, &\text{if $k=s$}.\end{cases}$$ So we can find $w\in W$ satisfying $w(\e_i-\e_j)=\e_{n-1}-\e_n$ and $w(\e_i+\e_j)=\e_{n-1}+\e_n$. By assumption we obtain  $(w(\mu+\rho), \e_{n-1}-\e_n)=1$ or,  equivalently, $$(w\cdot\mu, \e_{n-1}-\e_n)=0.$$ Using the previous lemma we have $$h(z)(w\cdot \mu)=h(z)(w\cdot \mu+t(\e_{n-1}+\e_n))$$ for all $t\in F_p$. Then we get from Lemma 3.7 that  $$\begin{aligned} h(z)(\mu)&=h(z)(w\cdot \mu+t(\e_{n-1}+\e_n))\\&=h(z)(w\cdot (\mu+t(\e_i+\e_j))\\ &=h(z)(\mu+t(\e_i+\e_j))\end{aligned}$$ for all $t\in  F_p$, as required.
\end{proof}
 For  $i\neq j$, set $s_{ij}=\{\mu\in\Lambda_0|(\mu+\rho, \e_i-\e_j)=1\}$.
\begin{lemma}Each $h(z)\in h(\mathcal Z)$ is constant on $\cup_{i\neq j}s_{ij}$.
\end{lemma}
\begin{proof}(cf. \cite[Lemma 4.4]{sv}).  By the proof of above lemma, we see that, for each $s_{ij}$ and each $\mu\in s_{ij}$, there is  $w\in W$ such that $w\cdot \mu\in s_{n-1,n}$. Since $h(z)\in u(\mathfrak h)^{W\cdot}$, it suffices to show that $h(z)$ is constant on $s_{12}$. This is immediate from Lemma 3.8 when $n=2$. \par Suppose $n>2$. Let $\mu=\sum_{i=1}^n\mu_i\e_i\in s_{12}$. Set $$\l=\mu+(\mu_3-\mu_1)(\e_1+\e_2)\  \text{and}\  \nu=\l +(\mu_1-\mu_3+1)(\e_2+\e_3).$$ Then we have $\l\in s_{23}$. By Lemma 3.9 we have $$h(z)(\mu)=h(z)(\l)=h(z)(\nu).$$ Let $w=s_{\e_1-\e_2}s_{\e_1-\e_3}\in W$. By a short computation we see that $w\cdot \nu=\mu+2\e_3$. Then we have $$\begin{aligned}h(z)(\mu)&=h(z)(\nu)\\ &=h(z)(w\cdot \nu)\\&=h(z)(\mu+2\e_3),\end{aligned}$$ and hence $h(z)(\mu)=h(z)(\mu +2k\e_3)$ for all $k\in F_p$. Since $p>2$,  we get $h(z)(\mu)=h(z)(\mu +t\e_3)$ for all $t\in F_p$. Similarly, we have $$h(z)(\mu)=h(z)(\mu+t\e_i)$$ for all $i>2$ and all $t\in F_p$. Therefore, $h(z)$ is constant on $s_{12}$.
\end{proof}
\begin{lemma} Let $h(z)\in h(\mathcal Z)$ be such that $h(z)|_{s_{n-1,n}}=0$. Then $(h_{n-1}-h_n)|h(z)$ in $u(\mathfrak h)$.
\end{lemma}
\begin{proof} Write  $$h(z)=\sum _{0\leq i,j\leq p-1}c_{ij}h_{n-1}^ih_n^j,$$ where each $c_{ij}\in u(\mathfrak h)$ is a polynomial of the variables $h_1,\dots, h_{n-2}$. In particular, if $n=2$, each $c_{ij}$ is just a scalar. By assumption, we have $h(z)(\mu)=0$ if $\mu\in \Lambda_0$ with $\mu_{n-1}=\mu_n$. \par
Note that $$h(z)=\sum_{s=0}^{p-1}(\sum_{i+j\equiv s} c_{ij}h_{n-1}^ih_n^j)$$ and  $k^{i+j}=k^s$ for all $k\in F_p$ if $i+j\equiv s$. Then we have for any $\mu\in s_{n-1,n}$ that $$\begin{aligned} h(z)(\mu)&=\sum_{s=0}^{p-1}(\sum_{i+j\equiv s}c_{ij}(\mu)\mu_n^{i+j})\\&=\sum_{s=0}^{p-1}(\sum_{i+j\equiv s}c_{ij}(\mu))\mu_n^s\\&=0,\end{aligned}$$
implying that  $\sum_{i+j\equiv s}c_{ij}(\mu_1,\dots,\mu_{n-2})=0$ for all $s\in F_p$.  It follows from Lemma 3.5 that $\sum_{i+j\equiv s}c_{ij}=0\in u(\mathfrak h)$ for $s$. Since $h_n^s=h_n^{p+s}\in u(\mathfrak h)$, we obtain $$\begin{aligned}\sum_{i+j\equiv s}c_{ij}h_{n-1}^ih_n^j&=\sum_{i+j\equiv s}c_{ij}h_{n-1}^ih_n^j-\sum_{i+j\equiv s}c_{ij}h_n^s\\
&=\sum_{i+j=s}c_{ij}(h_{n-1}^ih_n^j-h_n^s)+\sum_{i+j=p+s}c_{ij}(h_{n-1}^ih_n^j-h_n^{p+s})\\
&=\sum_{i+j=s}c_{ij}(h_{n-1}^i-h_n^i)h_n^{s-i}+\sum_{i+j=p+s}c_{ij}(h_{n-1}^i-h_n^i)h_n^{p+s-i},\end{aligned}$$
which gives $(h_{n-1}-h_n)|h(z)$, as required.

\end{proof}
For $i\neq j$, define $\Theta_{ij}\in u(\mathfrak h)$ by $$\Theta_{ij}(\mu)=(h_i-h_j)(\mu+\rho)-1=(\mu+\rho,\e_i-\e_j)-1,\  \mu\in \Lambda_0.$$ Then a short computation shows that $\Theta_{ij}=h_i-h_j+j-i-1$.\par
\begin{corollary} Let $i\neq j$ and let $h(z)\in h(\mathcal Z)$ be such that $h(z)|_{s_{ij}}=0$. Then $\Theta_{ij}|h(z)$
in $u(\mathfrak h)$.
\end{corollary}
\begin{proof} There is $w\in W$ such that $w(\e_i-\e_j)=\e_{n-1}-\e_n$. Let $\mu\in s_{ij}$. Since $(\mu+\rho, \e_i-\e_j)=1$, so that $(w(\mu+\rho),w(\e_i-\e_j))=1$, we have $(w\cdot\mu+\rho, \e_{n-1}-\e_n)=1$, and hence
$w\cdot \mu\in s_{n-1,n}$. Therefore, we get $w\cdot s_{ij}=s_{n-1,n}$.\par
Meanwhile, since $$\begin{aligned} \Theta_{ij}(\mu)&=(\mu+\rho,\e_i-\e_j)-1\\&=(w\cdot \mu+\rho, \e_{n-1}-\e_n)-1\\&=\Theta_{n-1,n}(w\cdot \mu)\\&=(w^{-1}\cdot \Theta_{n-1,n})(\mu)\ \text{for all} \ \mu
\in \Lambda_0,\end{aligned}$$ we have $\Theta_{ij}=w^{-1}\cdot \Theta_{n-1,n}$.\par
Since $w\cdot h(z)=h(z)$ and since $w^{-1}\cdot \l\in s_{ij}$ for $\l\in s_{n-1,n}$, we have by assumption that $$h(z)(\l)=h(z)(w^{-1}\cdot \l)=0.$$ Note that  $\Theta_{n-1,n}=h_{n-1}-h_n$.  Then  Lemma 3.11  gives
$\Theta_{n-1,n}|h(z)$, and hence $w^{-1}\cdot\Theta_{n-1,n}|w^{-1}\cdot h(z)$, implying that $\Theta_{ij}|h(z)$, as required.
\end{proof}
Let $z_i=h_i^p-h_i$ for $1\leq i\leq n$. By \cite[Theorem 5.1.2]{sf}, $$\mathcal O'=:F[z_1,\dots, z_n]\subseteq U(\mathfrak h)$$ is a polynomial ring. Clearly we see that $\mathcal O'=\mathcal O'_{F_p}\otimes_{F_p}F$, where $\mathcal O'_{F_p}=F_p[z_1,\dots, z_n]$. Then we have $$u(\mathfrak h)=U(\mathfrak h)/\mathcal O'U(\mathfrak h)=U(\mathfrak h_{F_p})\otimes_{F_p}F/(\mathcal O'_{F_p}U(\mathfrak h_{F_p}))\otimes_{F_p}F\cong u(\mathfrak h_{F_p})\otimes_{F_p} F.$$
\begin{lemma} $u(\mathfrak h)^{W\cdot}=u(\mathfrak h_{F_p})^{W\cdot}\otimes_{F_p}F$.
\end{lemma}
\begin{proof} Let $\{a_i\}_{i\in I}$ be a basis of the $F_p$-vector space $F$, and let $$x=\sum f_i(h)\otimes a_i\in (u(\mathfrak h_{F_p})\otimes F)^{W\cdot}.$$ Since $w\cdot x=x$ for every $w\in W$, we get $\sum _i (w\cdot f_i(h)-f_i(h))\otimes a_i=0$ for all $w$. Substituting $h$ by each $\l\in\Lambda_0$,  we get, for each $i$, $(w\cdot f_i(h)-f_i(h))(\l)=0$ for all $\l\in\Lambda_0$. Then Lemma 3.5 yields $w\cdot f_i(h)-f_i(h)=0$ for all $w\in W$, and hence $f_i(h)\in u(\mathfrak h_{F_p})^{W\cdot}$, as desired.
\end{proof}
\begin{theorem}  Assume $p>2d$. Then $h(\mathcal Z)\subseteq \Theta u(\mathfrak h)^W+ F$.
\end{theorem}
\begin{proof} Let $h(z)\in h(\mathcal Z)$.  By Lemma 3.10, there is $k\in F$ such that $(h(z)-k)|_{s_{ij}}=0\  \text{for all}\  i\neq j.$ Then  Corollary 3.12  gives  $\Theta_{ij}|(h(z)-k)$ in $u(\mathfrak h)$ for all $i\neq j$.
 It follows that $\Theta|(h(z)-k)^{2d}$ and hence $\Theta|(h(z)-k)^p$ in $u(\mathfrak h)$. \par

  Assume $(h(z)-k)^p=\Theta q(h)$ for some $q(h)\in u(\mathfrak h)$. Since $(h(z)-k)^p, \Theta\in u(\mathfrak h)^{W\cdot}$, we have $(h(z)-k)^p=\Theta\ (w\cdot q(h))$ for every $w\in W$. Since $p>2d=n(n-1)\geq n$,  $|W|=n!\neq 0$. Let $$ \tilde{q}(h)={1\over |W|}\sum_{w\in W} w\cdot q(h).$$ Then we have $\tilde{q}(h)\in u(\mathfrak h)^{W\cdot}$ and  also $(h(z)-k)^p=\Theta\tilde{q}(h)$. Thus, we may assume that $q(h)\in u(\mathfrak h)^{W\cdot}$.\par
 Let $\{a_i\}_{i\in I}$ be a basis of the $F_p$-vector space $F$ as above. Since $F$ is algebraically closed, the set $\{a_i^p\}_{i\in I}$ is also a basis. In fact, for each $a\in F$, we have $a^{1/p}=\sum_i \l_i a_i$ for $\l_i\in F_p$, so that $a=\sum \l_i a_i^p$. The linear independence of the set $\{a_i^p\}$ is immediate.\par
By the lemma above, we have $h(z)-k=\sum_i h_i(z)\otimes a_i$, where $h_i(z)\in u(\mathfrak h_{F_p})^{W\cdot}$. Then  $$(h(z)-k)^p=\sum_i h_i(z)^p\otimes a_i^p=\sum_i h_i(z)\otimes a_i^p.$$ Assume $q(h)=\sum_i q_i(h)\otimes a_i^p$ with $q_i(h)\in u(\mathfrak h_{F_p})^{W\cdot}$. Then we get $$\sum_i(h_i(z)-\Theta q_i(h))\otimes a_i^p=0.$$ By a similar argument as that used in the proof of above lemma,  we have $h_i(z)=\Theta q_i(h)$ for all $i$, implying that $\Theta|(h(z)-k)$.  Let $h(z)-k=\Theta q'(h)$. By the discussion above we may assume $q'(h)\in  u(\mathfrak h)^{W\cdot}$. This completes the proof.

\end{proof}

\subsection{The image of $h$}
In this subsection, we construct  some central elements in $\bar{\mathfrak u}$.\par
Set $\wedge^{\leq d} (\g_1)=:\sum_{i=0}^d\wedge ^i(\g_1)$.
 From Section 3.2, we obtain an isomorphism $\bar{\mathfrak u}\cong \wedge^{\leq d}(\g_1)\otimes u(\g_{\0})\otimes \wedge (\g_{-1})$ as vector spaces.
Then  $\bar{\mathfrak u}$ has a basis of the form
 $$z^{\d'}x^sy^{\d},\ z^{\d'}\in \wedge^{\leq d} (\g_1),\ x^s\in u(\g_{\0}), \ y^{\d}\in\wedge (\g_{-1}).$$
 Define a projection $(,)_0$: $\bar {\mathfrak u}\longrightarrow u(\g_{\0})$
 by $$(z^{\d'}x^sy^{\d})_0=\begin{cases} x^s, &\text{if $z^{\d'}=y^{\d}=1$}\\
                                             0,   &\text{otherwise.}\end{cases}$$

  Recall the notion $Y$ and $Z$ defined earlier.
\begin{lemma} Assume $p>2d$. Then $(Y\wedge^d(\g_1))_0\neq 0$.
\end{lemma}
\begin{proof} Let $L(\mu)$ be the restricted simple $\g_{\0}$-module of highest weight $\mu\in\Lambda_0$.
 Note that $\delta_{\mu}=(\Pi_{i<j}\Theta_{ij})(\mu)$ for all $\mu\in \Lambda_0$. Since $p>n(n-1)$,  the degree of each $h_i$ in the polynomial $\Pi_{i<j}\Theta_{ij}$ is less that $p$. By Lemma 3.5, there is $\mu\in\Lambda_0$ such that $(\Pi_{i<j}\Theta_{ij})(\mu)\neq 0$.  It then follows from Proposition 2.3 that $K(L({\mu}))$ is  simple as a $u(\g)$-module.\par
  Let $v_{\mu}\in L(\mu)$ be a maximal vector. Write $YZ$ as a linear combination  of the above basis for $\bar {\mathfrak u}$. Then we have in $K(L(\mu))$ that $$(YZ)_0Yv_{\mu}=
YZYv_{\mu}=\delta_{\mu}Yv_{\mu}\neq 0,$$  implying that $(YZ)_0\neq 0$.
\end{proof}
  Since   $\wedge (\g_{-1})$, $\wedge^{\leq d} (\g_1)$ and $u(\g_0)$ are $G$-submodules of $u(\g)$, it is easy to see that $(,)_0$ is a homomorphism of $G$-modules.\par
Let $u^d(\g_{\0})=\langle x_1\cdots x_k\in u(\g_{\0})| x_1,\dots, x_k\in\g_{\0},\  k\leq d\rangle$. Clearly  $u^d(\g_{\0})$ is a $G$-submodule of $u(\g)$.
\par
Define a linear map $\varphi: \wedge^d (\g_1)\longrightarrow u^d(\g_{\0})$ by letting $\varphi(x)=(Yx)_0$ for  $x\in \wedge^d(\g_1)$.

\begin{lemma} $\varphi$ is a homomorphism of $G'$-modules.
\end{lemma}
\begin{proof}Recall that  $G$ acts on $\g$ by the composition of the adjoint action with $\rho$. A short computation shows that,  for $c\in F$, $s<t$, $$Ad\rho(1+ce_{ij}) y_{st}=\begin{cases}& y_{st}-cy_{jt}, \  \text{if $i=s$}\\
&y_{st},\ \  \text{otherwise.}\end{cases}  $$ Then we get $(1+ce_{ij})\cdot Y=Y$. Since  $G'=SL_n(F)$ is generated by all $1+ce_{ij}$, $i\neq j$(\cite[Lemma 6.7.1] {ja}), we have $g\cdot Y=Y$ for all $g\in G'$.\par
For any $x\in\wedge^d(\g_1)$ and $g\in G'$,  we then have $$\begin{aligned}\varphi(g\cdot x)&=(Y\rho(g)x\rho(g)^{-1})_0\\
&=(\rho(g)(Yx)\rho(g^{-1}))_0\\&=(g\cdot(Yx))_0\\&=g\cdot(Yx)_0\\&=g\cdot\varphi(x),\end{aligned}$$ as required.
\end{proof}

\subsubsection{A simple $\g_{\0}$-module of typical weight}
In the following we study a simple $\g_{\0}$-module  used later to construct central elements in $\bar{\mathfrak u}$.
 \begin{definition} A weight $\l\in \Lambda_0$  is {\sl typical} if $\Theta(\l)\neq 0$ and {\sl atypical}  otherwise.
\end{definition}
Example. Let $\l=(n-1)\e_1+(n-2)\e_2+\cdots +\e_{n-1}\in \Lambda_0$. If $p>2n-1$, then the weight $\l$ is typical.\par
Note that if $\l$ is typical, then $\delta_{\l}\neq 0$.\par
Denote by $\g_{\mathbb C}$ the Lie superalgebra $p(n)$  over $\mathbb C$. Let $V(\tilde\l)$ be the simple $\g_{\0, \mathbb C}$-module with highest weight $\tilde\l=(n-1)\e_1+(n-2)\e_2+\cdots +\e_{n-1}\in \mathfrak h^*$.

  Since $\tilde\l$ is dominant integral, $V(\tilde\l)$ is finite dimensional (see \cite[Theorem 21.2]{hu1}). Assume $\Phi_0^+=\{\a_1,\dots, \a_d\}$, where $\a_i=\e_i-\e_{i+1}$ for $1\leq i\leq n-1$. Then $\beta=\a_1+\cdots +\a_{n-1}\in\Phi^+_0$
 is the longest root. For each $\a=\e_i-\e_j\in \Phi^+_0$, let $e_{\a}=\tilde e_{ij}, f_{\a}=\tilde e_{ji}$ and $h_{\a}=h_i-h_j$. \par Let $v^+\in V(\tilde\l)$ be the maximal vector.  Then $V(\tilde\l)$ is spanned by vectors of the form $f_{\a_1}^{s_1}\cdots f_{\a_d}^{s_d}v^+$, $s_i\geq 0$.  Let $\mathfrak{sl}^{\beta}_2=\langle e_{\beta}, f_{\beta}, h_{\beta}\rangle \subseteq \g_{\0, \mathbb C}$. By \cite[Theorem 6.3]{hu1}, $V(\tilde\l)$ is completely reducible as a $\mathfrak{sl}^{\beta}_2$-module. Moreover, each simple summand  has  maximal weight $m\geq 0$ and has dimension $m+1$ (\cite[Theorem 7.2]{hu1}).\par  We claim that $\tilde{\l}(h_{\beta})$ is the largest of all these weights. In fact, note that  all  weights (with respect to $\mathfrak h_{\mathbb C}=\langle h_1,\cdots, h_n\rangle$) in $V(\tilde\l)$ are of the form $$\tilde\l-\sum_{i=1}^{n-1} {k_i}\a_i,\  k_i\geq 0.$$  Since  $\a_i=\e_i-\e_{i+1}$ for all $i$ and $h_{\beta}=h_1-h_n$,  we have $\a_i(h_{\beta})=\delta_{i1}+\delta_{i+1,n}\geq 0$. Then the claim follows.\par
Since $\tilde\l (h_{\beta})=n-1$, it follows from \cite[Theorem 7.2]{hu1} that every simple summand of the $\mathfrak{sl}^{\beta}_{2}$-module $V(\tilde\l)$ has dimension $\leq n$. Thus, we obtain that $f_{\beta}^nV(\tilde\l)=0$, and similarly $e_{\beta}^n V(\tilde\l)=0$.
\begin{lemma} Let $\a\in \Phi_0^+$.  Then $e_{\a}^nV(\tilde\l)=f_{\a}^nV(\tilde\l)=0$.
\end{lemma}
\begin{proof} Note that $\Phi_0^+=\{\e_i-\e_j| 1\leq i<j\leq n\}$. For each $\a\in \Phi_0^+$, define $$x_{\a}(t)=\text{exp}(tad e_{\a}),\  t\in \mathbb C.$$ Then $x_{\a}(t)$ is an automorphism of $\g_{\0, \mathbb C}$. Set $\tilde \sigma_{\a}=x_{\a}(1)x_{-\a}(-1)x_{\a}(1)$. According to \cite[Lemma 6.4.4]{cr} and \cite[Proposition 6.4.2]{cr}, we have $\tilde \sigma_{\a}(h_{\beta})=h_{\sigma_{\a}(\beta)}$ and $\tilde \sigma_{\a}(e_{\beta})=\pm e_{\sigma_{\a}(\beta)}$ for any $\beta\in \Phi_0$, where $\sigma_{\a}$ is the refection defined with $\a$ (see \cite[Section 9.1]{hu1}).\par
For each $\a\in \Phi_0^+$, we have a triangular decomposition of $\g_{\0, \mathbb C}$ with respect to $\a$: $$\g_{\0, \mathbb C}=\tilde \sigma_{\a}(\mathfrak n_0^+)+\mathfrak h+\tilde \sigma_{\a}(\mathfrak n_0^-).$$ The corresponding positive root system(resp. bases) is $\sigma_{\a}(\Phi_0^+)$ (resp. $\sigma_{\a}(\Delta)$). Since $e_{\a}$ acts  nilpotently on $V(\tilde\l)$, we may define an automorphism of the vector space $V(\tilde\l)$ by (see \cite[Chapter 8]{jj2} or \cite[21.25]{hu1})  $$\sigma_{\a,V}=\text{exp}(e_{\a,V})\text{exp}(-e_{-\a,V})\text{exp}(e_{\a,V}),$$ where $e_{\a,V}$ denotes the endomorphism of $V(\tilde\l)$ given by the action of $e_{\a}$. Then for $x\in\g_{\0}$ and $v\in V(\tilde\l)$, we have $$\sigma_{\a, V}(xv)=\tilde \sigma_{\a}(x) \sigma_{\a,V}(v).$$ This follows from \cite[Lemma 4.5.2]{cr} (for the case $\a$ is simple, see also \cite[Chapter 8]{jj2}).\par
Let $\mathfrak b=\mathfrak h+\mathfrak n_0^+ $.  We claim that $\sigma_{\a,V} (v^+)\in V(\tilde\l)$ is the maximal vector with respect to the Borel subalgebra $\tilde\sigma_{\a}(\mathfrak b)$. In fact, we have $$\tilde\sigma_{\a}(x)\sigma_{\a,V} (v^+)=\sigma_{\a,V}(xv^+)=0$$ for all $x\in \mathfrak n_0^+$ and $$\begin{aligned} h\sigma_{\a,V}(v^+)&=\sigma_{\a,V}(\tilde\sigma_{\a}^{-1}(h)v^+)\\ &=\tilde{\l}(\tilde\sigma_{\a}^{-1}(h))\sigma_{\a,V}(v^+)\\ &=(\sigma_{\a} \tilde\l)(h)\sigma_{\a,V}(v^+),\end{aligned}$$ for $h\in\mathfrak h$. Then the claim follows. In addition, we see that the highest weight of $V(\tilde\l)$ is $\sigma_{\a} \tilde\l$.\par
For each $\a=\e_i-\e_j\in \Phi_0^+$,  let $\a_1=\e_1-\e_i$ and $\a_2=\e_j-\e_n$. We let $\sigma_{\a_2}=1$ if $j=n$ and $\sigma_{\a_1}=1$ if $i=1$. Then we have $\sigma_{\a_2}\sigma_{\a_1}(\beta)=\a$, where $\beta=\e_1-\e_n$ is the longest root. \par By above discussion,  the highest weight of $V(\tilde\l)$ with respect to the Borel subalgebra $\tilde\sigma_{\a_2}\tilde\sigma_{\a_1}(\mathfrak b)$ is $\sigma_{\a_2}\sigma_{\a_1}\tilde\l$. Meanwhile, the longest root in the positive system $\sigma_{\a_2}\sigma_{\a_1}(\Phi_0^+)$ is $\a$. Decompose $V(\tilde\l)$ as a direct sum of simple $\mathfrak{sl}^{\a}_2$-submodules. By a similar argument as above, we obtain that
 the largest weight of these simple summands is $$\begin{aligned}(\sigma_{\a_2}\sigma_{\a_1}\tilde\l)(h_{\a})
 &=(\sigma_{\a_2}\sigma_{\a_1}\tilde\l)(h_{\sigma_{\a_2}\sigma_{\a_1}(\e_1-\e_n)})\\
 &=(\sigma_{\a_2}\sigma_{\a_1}\tilde\l)(\tilde\sigma_{\a_2}\tilde\sigma_{\a_1}(h_{\e_1-\e_n}))\\
 &=\tilde{\l}(h_{\e_1-\e_n})\\&=n-1.\end{aligned}$$
 This implies that $f_{\a}^nV(\tilde\l)=0$ and similarly $e_{\a}^nV(\tilde\l)=0$.
\end{proof}
 Set $d_1=n^2(n-1)$. Let $$U^{d_1}(\g_{\0,\mathbb C})=\langle x_1\cdots x_k\in U(\g_{\0, \mathbb C})| x_1,\dots, x_k\in \g_{\0}, k\leq d_1\rangle.$$ With the assumption $p>d_1$, the notation $u^{d_1}(\g_{\0})$($\subseteq u(\g)$) can be defined similarly.

\begin{lemma} Let $\vartheta: U(\g_{\0, \mathbb C})\longrightarrow \text{End}_{\mathbb C} V(\tilde\l)$ be the representation afforded by the $\g_{\0, \mathbb C}$-module $V(\tilde\l)$. Then $\vartheta (U^{d_1}(\g_{\0, \mathbb C}))=\text{End}_{\mathbb C} V(\tilde\l)$.
\end{lemma}
\begin{proof} Set $\text{ann}(V(\tilde\l))=\{u\in U(\g_{\0, \mathbb C})|u V(\tilde\l)=0\}$. Then $\text{ann} V(\tilde\l)$ is an ideal of $U(\g_{\0})$.   Moreover, $\bar U(\g_{\0, \mathbb C})=U(\g_{\0, \mathbb C})/\text{ann} (V(\tilde\l))$ is primitive and $V(\tilde\l)$ is a faithful $\bar U(\g_{\0, \mathbb C})$-module. By \cite[Theorem 1.12, p. 420]{hunw}, $\vartheta(\bar U(\g_{\0,\mathbb C}))$ is dense in $\text{End}_{\mathbb C}V(\tilde\l)$, implying that  $\vartheta(U(\g_{\0, \mathbb C}))=\text{End}_{\mathbb C} V(\tilde\l)$.\par
From the proof of above lemma, we see that, for $\a\in\Phi_0^+$, each simple $\mathfrak{sl}^{\a}_2$-submodule   of $V(\tilde\l)$ has dimension $\leq n$. Then  $\vartheta(h_{\a})$ satisfying either $$(\vartheta(h_{\a})-(n-1))(\vartheta(h_{\a})-(n-3))\cdots (\vartheta(h_{\a})+(n-1))=0$$ or $$(\vartheta(h_{\a})-(n-2))(\vartheta(h_{\a})-(n-4))\cdots (\vartheta(h_{\a})+(n-2))=0$$ or both. Hence $\vartheta(U(\mathfrak h))$ is spanned by the images of $h_1^{k_1}\cdots h_n^{k_n}, 0\leq k_i\leq n-1$.
It follows that $\text{End}_{\mathbb C} V(\tilde\l)$ is spanned by the images under $\vartheta$ of the  elements
$$e_{\a_1}^{s_1}\cdots e_{\a_d}^{s_d}f_{\a_1}^{t_1}\cdots f_{\a_d}^{t_d}h_1^{k_1}\cdots h_n^{k_n},\  0\leq s_i,t_i,k_i\leq n-1.$$  All these elements are contained in $U^{d_1}(\g_{\0})$, then the lemma follows.
\end{proof}

  Let $\g_{\0, \mathbb Z}$ be the $\mathbb Z$-span of the  basis vectors $\tilde{e}_{ij}$ of $\g_{\0, \mathbb C}$.   Let $\mathcal U_{\mathbb Z}$ and $\mathcal U'_{\mathbb Z}$ denote respectively the Kostant $\mathbb Z$-form of $U(\g_{\0, \mathbb C})$ and $U(\g'_{\0, \mathbb C})$ (see \cite[26]{hu1}). Let $\hbar=\sum_{i=1}^n h_i\in\g_{\0, \mathbb Z}$. Then $[\hbar,\ \g_{\0,\mathbb C}]=0$. Let $\mathcal U(\hbar)_{\mathbb Z}$ be the $\mathbb Z$-submodule of $\mathcal U_{\mathbb Z}$ generated by all $\binom{\hbar}{a}, a\geq 0$. Then we have $\mathcal U_{\mathbb Z}=\mathcal U'_{\mathbb Z}\otimes_{\mathbb Z}\mathcal U(\hbar)_{\mathbb Z}$.   \par

  Let  ${\bar V}(\l)=\mathcal U_{\mathbb Z}v^+\otimes F$. Then ${\bar V}(\l)$
   is a rational $G$-module (hence restricted $\g_{\0}$-module). Since $p>n$,  the maximal vector $v^+\otimes 1\in {\bar V}(\l)$ has the restricted weight $\l=\tilde\l\otimes 1\in \Lambda_0$.  Clearly we have $\bar V(\l)=\mathcal U'_{\mathbb Z}v^+\otimes F$. By \cite[Proposition 1.2]{hu2}, the $\g_{\0}'$-module $\bar V(\l)$ has a unique maximal submodule. It follows that $\bar V(\l)$, as a $\g_{\0}$-module, also has a unique maximal submodule. Therefore, $\bar V(\l)$ has a unique quotient $L(\l)$ simple as a restricted $\g_{\0}$-module (hence as a $G$-module).\par

\begin{theorem} Let $\vartheta$ be the representation of $u(\g_{\0})$ afforded by $\g_{\0}$-module $L(\l)$.  Then  $\vartheta(u^{d_1}(\g_{\0}))=\text{End}_F L(\l)$.
\end{theorem}
\begin{proof} By Lemma 3.18 we have $e_{\a}^n\bar V(\l)=f_{\a}^n\bar V(\l)=0$
 and hence $e_{\a}^nL(\l)=f_{\a}^nL(\l)=0$ for all $\a\in \Phi_0^+$. From the proof of Lemma 3.19,
  we see that $L(\l)$ is annihilated by $$\Pi_{k=-(n-1)}^{n-1}(\vartheta(h_{\a})-k)$$ for each $\a\in \Phi_0^+$.
 Then the theorem follows from a similar argument as that used in the proof of Lemma 3.19.
\end{proof}

\subsubsection{An $\text{ad}_{\g'_{\0}}$-submodule of $u^{d_1}(\g)$}
In the following, we continue making the preparation  for constructing central elements in $\bar{\mathfrak u}$.\par
 Let $U(\g_Q)$ be the universal enveloping superalgebra of the Lie superalgebra $p(n)$  over $Q$. Set $$V_Q=:U^{d_1}(\g_{\0, Q})+\wedge^d(\g_{1, Q}).$$ Then $V_Q$ is an $\text{ad}_{\g'_{\0, Q}}$-submodule of $U^{d_1}(\g_Q)$. Since $\g_{\0, Q}'\cong \mathfrak{sl}(n, Q)$ is semisimple,   we have by \cite[Theorem 6.3]{hu1} that $V_Q=V_1\oplus V_2\oplus \cdots \oplus V_s$, where each $V_i$ is a simple $\g'_{\0, Q}$-module.

\begin{lemma} Each  $V_i$ is generated by a maximal vector $v_{\l}\in V_i$ of dominant integral weight $\l$.
\end{lemma}
\begin{proof} Since $\mathbb C$ is algebraically closed and since $\g'_{\0}$ is semisimple,  $V_i\otimes_Q\mathbb C$ is a direct sum of simple $\g'_{\0, \mathbb C}$-modules each of the form $V(\l)$  with $\l$ dominant integral (see \cite[21.2]{hu1}).\par  Let $V(\l)$ be one of the summands,  let $v_{\l}\in V(\l)$ be the maximal vector and let $\{\xi_i| i\in I\}$ be a basis of $\mathbb C$ over $Q$. Write $v_{\l}=\sum_{j=1}^t \xi_jv_j$,  $v_j\in V_i$.\par
If $t=1$, we let $v_1=v_{\l}\in V_i$. Then $U(\g'_{\0, Q})v_{\l}\subseteq V_i$ is a $\g'_{\0, Q}$-submodule of $V_i$, forcing $V_i=U(\g'_{\0, Q})v_{\l}$, as required.\par
If $t>1$, then for all $\a\in \Phi_0^+$, we have $e_{\a}v_{\l}=0$ and hence $\sum^t_{j=1}\xi_j(e_{\a}v_j)=0$.  Applying every $$f\in \text{Hom}_Q(V_i,Q)\subseteq \text{Hom}_{\mathbb C}(V_i\otimes_Q \mathbb C, \mathbb C)$$ to this equation,  we have $\sum_{j=1}^t \xi_jf(e_{\a}v_j)=0$, and hence $f(e_{\a}v_j)=0$ for each $j$. Therefore, for each $j$, we have  $e_{\a}v_j=0$ for all $\a\in\Phi_0^+$. Similarly,  we obtain $h_kv_j=\l(h_k)v_j$ for all $k=1,\dots, n$. It follows that each $v_j$ is a maximal vector of weight $\l$. So that $U(\g'_{\0, Q})v_j\subseteq V_i$ is a $\g'_{\0, \mathbb C}$-submodule, implying that $V_i=U(\g'_{\0, Q})v_j$, as required.
\end{proof}
Fix a PBW basis of $U(\g_{\mathbb C})$ consisting of monomials of the natural basis of $\g_{\mathbb C}$.
Let $U(\g_{\mathbb Z})$ denote the $\mathbb Z$-span of this basis, and let $V_{\mathbb Z}$ be the $\mathbb Z$-span of the basis vectors contained in $V_Q$.\par
  By the lemma above, each $V_i$ may be written as $V(\l_i)$ as in \cite[21.1]{hu1}. In addition, we may assume  $v_{\l_i}=\sum_{i=1}^r c_iv_i\in V_{\mathbb Z}$ with each $v_i$ a PBW basis vector,  each $c_i\in\mathbb Z$  and  $(c_1,c_2,\dots, c_r)=1$.  Then $\mathcal U'_{\mathbb Z}v_{\l_i}$ is an admissible lattice in
$V(\l_i)$. According to the proof of \cite[Theorem 27.1]{hu1}, $\mathcal U'_{\mathbb Z}v_{\l_i}$ is a free $\mathbb Z$-module with $\mathbb Z$-rank   $\text{dim}V(\l_i)$. Moreover, we have $\mathcal U'_{\mathbb Z}v_{\l_i}\subseteq V_{\mathbb Z}$  by \cite[Corollary 26.3]{hu1}.\par
 Set $\bar V(\l_i)=\mathcal U'_{\mathbb Z}v_{\l_i}\otimes F$. This is a rational $G'$-module and hence a restricted $\g'_{\0}$-module (see \cite{hu2}). Then the inclusion $\mathcal U'_{\mathbb Z}v_{\l_i}\subseteq V_{\mathbb Z}$ induces a $G'$-module homomorphism from $\bar V(\l_i)$ to $V=:V_{\mathbb Z}\otimes F$.\par
 Remark:  Since $p>2$, we have $U(\g_{\mathbb Z})\otimes F\cong U(\g)$, which gives $V\cong U^{d_1}(\g_{\0})\oplus \wedge ^d(\g_1)$.\par
 Now that $\mathcal U'_{\mathbb Z}v_{\l_1}\oplus \cdots \oplus\mathcal U'_{\mathbb Z}v_{\l_s}\subseteq V_{\mathbb Z}$ is a free $\mathbb Z$-module of $\mathbb Z$-rank $$\sum_{i=1}^s\text{dim} V(\l_i)=\text{dim} V_Q.$$
 For each $V(\l_i)$,  there is a contravariant bilinear form whose determinant on $\mathcal U'_{\mathbb Z}v_{\l_i}$ is an integer $D_{\l_i}$ (\cite{jj2}). The  induced  bilinear form on $\bar V(\l_i)$ has as the radical the unique maximal $G'$-submodule.\par  Assume $(C1)$: $p\nmid  \Pi^s_{i=1} D_{\l_i}$. Then the contravariant  form on each $\bar V(\l_i)$ is nondegenerate, and hence each $\bar V(\l_i)$ is a simple $\g_{\0}'$-module (and a simple $G'$-module).
 Therefore, we obtain a completely reducible $G'$-submodule  $\bar V(\l_1)+ \cdots + \bar V(\l_s)$ of $V$.\par
 Let $v_1,\dots, v_l$ denote the fixed PBW basis of $V_{\mathbb Z}$. Then we have $$(*)\\ \hskip 0.5truept\ (v_{\l_1}, \dots, v_{\l_s})=(v_1,\dots, v_l)A_{l\times s},$$ where $A_{l\times s}$ is an integral matrix. For each $v\in V_{\mathbb Z}$, denote $v\otimes 1\in V$ by $\bar v$. Denote the image of $A_{l\times s}$ under the canonical map  $\mathbb Z^{r\times s}\longrightarrow F_p^{r\times s}$ by $\bar A_{l\times s}$.
 It follows that   $$(\bar v_{\l_1}, \dots, \bar v_{\l_s})=(\bar v_1,\dots, \bar v_l)\bar A_{l\times s}.$$
 Since $v_{\l_1}, \dots, v_{\l_s}$ are linearly independent,   there is a  submatrix $A'_{s\times s}$ in $A_{l\times s}$ with $\text{det} A'=m\neq 0$. Assume $(C1'): p>|m|$. Then we have $r(\bar A_{l\times s})=s$, and hence the elements $\bar v_{\l_1}, \dots, \bar v_{\l_s}\in V$ are linearly independent.\par
\begin{corollary} With  assumptions $(C1)$ and $(C1')$, we have
 $$\bar V(\l_1)\oplus \cdots \oplus \bar V(\l_s)=V.$$
\end{corollary}
\begin{proof} Since both sides have the same dimension, it suffices to show that $\bar V(\l_1)+\cdots +\bar V(\l_s)\subseteq V$ is a direct sum.\par
 Assume $x_1+\cdots +x_s=0$, $x_i\in \bar V(\l_i)$, $1\leq i\leq s$. Since each $\bar V(\l_i)$ is a simple  $G'$-module, we may assume that each $x_i$ is a weight vector with respect to the maximal torus $T\cap G'$. \par Suppose that there is $i$, $1\leq i\leq s$, such that $x_i\neq 0$.
 If $\text{wt} (x_i)<\l_i$, then there is  $\a\in \{\e_1-\e_2,\dots, \e_{n-1}-\e_n\}$ such that $e_{\a} x_i\neq 0$. Clearly we have $\text{wt} (x_i)<\text{wt} (e_{\a}x_i)\leq \l_i$. Since each $\bar V(\l_i)$ has finite many nonzero weight spaces,
 by repeated applications of appropriate $e_{\a}$'s we have $$x_1'+\cdots +x_s'=0, \  x'_i\in \bar V(\l_i),$$ where $x'_i$'s are not all zero and   each $x_i$ is a multiple of $\bar v_{\l_i}$, contrary to the fact that $\bar v_{\l_1}, \dots, \bar v_{\l_s}$ are linearly independent. \par
 Thus, we conclude that $x_i=0$ for all $1\leq i\leq s$, which completes the proof.
\end{proof}
 With the assumption $p>d_1$,   we have an isomorphism of rational $G$-modules (hence of restricted $\g_{\0}$-modules) $U^{d_1}(\g_{\0})\cong u^{d_1}(\g_{\0})$, and hence $V\cong u^{d_1}(\g_{\0})\oplus \wedge ^d(\g_1)$. Together with assumptions $(C1)$ and $(C1')$,
 we obtain an isomorphism of rational $G'$-modules
$$\bar V(\l_1)\oplus \cdots \oplus \bar V(\l_s)\cong u^{d_1}(\g_{\0})\oplus \wedge ^d(\g_1).$$ Therefore,  both $u^{d_1}(\g_{\0})$ and $\wedge ^d(\g_1)$ are completely reducible as  restricted $\text{ad}_{\g'_{\0}}$-modules.\par

In the remainder of the paper we assume the following conditions:\par  (H1) $(C1)$ and  $(C1')$;\par  (H2) $p>d_1$;\par (H3) $p>\text{dim}\wedge^d (\g_1)=C^d_{\frac{1}{2}n(n+1)}$.\par

Under these assumptions, since $d_1>2d$, the condition on $p$ in Theorem 3.14 is satisfied.

\subsubsection{Central elements in $\bar{\mathfrak u}$}
Recall the weight $\l=(n-1)\e_1+\cdots +\e_{n-1}\in\Lambda_0$. From 3.5.1 we see that $\l$ is typical, since $p>d_1=(n-1)n^2\geq n^2\geq 2n-1$.\par
 Let $\vartheta$ be the representation of $u(\g_{\0})$ afforded by  the simple G-module $L(\l)$ in 3.5.1.   According to Theorem 3.20,  $\vartheta$ is an epimorphism from $u(\g_{\0})$ to $\text{End}_F L(\l)$. This induces  an isomorphism of algebras $$u(\g_{\0})/\text{ker} \vartheta\cong \text{End}_F L(\l).$$  Denote the restriction of $\vartheta$ to $u^{d_1}(\g_{\0})$ by $\vartheta_1$. Then  $\text{ker}\vartheta_1=\text{ker}\vartheta\cap u^{d_1}(\g_{\0})$, and  which is easily seen to be an $\text{Ad}G$-submodule  of $u^{d_1}(\g_{\0})$. By Theorem 3.20 we obtain an isomorphism of $G$-modules $$u^{d_1}(\g_{\0})/\text{ker} \vartheta_1\cong \text{End}_F L(\l).$$\par
 Recall the mapping $\varphi$. Let $v_{\l}\in L(\l)$ be the maximal vector. Since $\l$ is typical, we have $$(YZ)v_{\l}=(YZ)_0v_{\l}=\d_{\l}v_{\l}\neq 0,$$ implying that $\text{Im}\varphi\varsubsetneq \text{ker}\vartheta$. Since $\text{Im}\varphi\subseteq u^d(\g_{\0})$ and $d<d_1$,  we get $\text{Im}\varphi \varsubsetneq\text{ker}\vartheta_1$.
  Since  $u^{d_1}(\g_{\0})$ is a completely reducible  $\g_{\0}'$-module, so is $\text{Im}\varphi$.   It follows that  $L(\mu)\varsubsetneq\text{ker}\vartheta_1$ for some simple summand $L(\mu)$ of $\text{Im}\varphi$. \par Since $L(\mu)$ is simple,  we have $L(\mu)\cap \text{ker}\vartheta_1=\{0\}$.
Since $u^{d_1}(\g_{\0})$ is  completely reducible, we have $$u^{d_1}(\g_{\0})=(L(\mu) +\text{ker}\vartheta_1)\oplus W$$ for some $G'$-submodule $W$.\par Let $T_1$ be the toral subgroup of $G$ defined by $$T_1=\{diag(t,\dots, t)|t\in F^*\}.$$ Then we have $G=G'\times T_1$. Since $T_1$ acts trivially on $u(\g_{\0})$,  every $G'$-submodule of $u(\g_{\0})$ is naturally a $G$-submodule.  Then we have a $G$-module isomorphism $$u^{d_1}(\g_{\0})/\text{ker}\vartheta_1\cong L(\mu)+W\cong \text{End}_F L(\l).$$ Since the trace form on $\text{End}_F L(\l)$ is $G$-invariant, nondegenerate and symmetric, so is the induced  form on $L(\mu)+W$.\par
Since $\wedge^d(\g_1)$ is completely reducible, we have an isomorphism of $u(\g_{\0}')$-modules (also $G'$-modules) $$\text{ker}\varphi \oplus \text{Im}\varphi\cong \wedge^d(\g_1).$$ Let $Z_1,\dots, Z_r$ be a basis of the simple $u(\g'_{\0})$-submodule of $\wedge^d(\g_1)$ whose image under $\varphi$ is $L(\mu)$.
 Then $(YZ_1)_0, \dots, (YZ_r)_0$ is a basis of $L(\mu)$, from which we extend to a basis of $L(\mu)+W$: $$(YZ_1)_0, \dots, (YZ_r)_0, w_1,\dots, w_t.$$
Let $(YZ_1)_0^{\lor}, \dots, (YZ_r)_0^{\lor}, w_1^{\lor},\dots, w_t^{\lor}$ be the dual basis.
Set $$W^{\perp}=\{x\in L(\mu)+W| (x,W)=0\}.$$ Since $W$ is a $G$-submodule of $u(\g_{\0})$, so is  $W^{\perp}$. By \cite[Section 6.1]{ja}, we have $$\text{dim} W^{\perp}=\text{dim} (L(\mu)+W)-\text{dim} W=\text{dim}L(\mu)=r,$$
 implying that $W^{\perp}=\langle (YZ_1)_0^{\lor},\dots, (YZ_r)_0^{\lor}\rangle$.\par

For  $g\in G'$, assume  $$g(Z_1,\dots, Z_r)^t=(b^g_{ij})_{r\times r} (Z_1,\dots, Z_r)^t$$ and $$g((YZ_1)_0^{\lor}, \dots, (YZ_r)_0^{\lor})^t=(c_{ij}^g)_{r\times r} ((YZ_1)_0^{\lor}, \dots, (YZ_r)_0^{\lor})^t.$$ Then since $Y$ is a $G'$-invariant, we have $$g(YZ_1,\dots, YZ_r)^t=(b^g_{ij})_{r\times r} (YZ_1,\dots, YZ_r)^t.$$ Since the mapping $(,)_0$ is a $G$-module homomorphism, we obtain $$g((YZ_1)_0, \dots, (YZ_r)_0)^t=(b_{ij}^g)_{r\times r} ((YZ_1)_0,\dots, (YZ_r)_0)^t.$$ Then the $G$-invariance of the bilinear form implies  that $(b_{ij}^g)_{r\times r}^t=(c_{ij}^g)_{r\times r}^{-1}.$\par
 Set $\omega=\sum_{i=1}^r Z_i(YZ_i)_0^{\lor}\in u(\g)$. The following conclusion follows from a standard argument in linear algebra.
\begin{lemma} For each $g\in G'$, we have $g\omega=\omega$.
\end{lemma}
Then by \cite[7.11(5), Part I]{jj1}, we have $[x, \omega]=0$ for all $x\in \g_{\0}'$.\par
\begin{lemma} $(Y\omega)_0\neq 0$.
\end{lemma}
\begin{proof} Since $(Y\omega)_0=\sum_{i=1}^r (YZ_i)_0(YZ_i)_0^{\lor}$, we have $$tr \vartheta(Y\omega)_0=\sum_{i=1}^r tr \vartheta((YZ_i)_0)\vartheta((YZ_i)_0^{\lor})=\sum_{i=1}^r((fZ_i)_0, (fZ)_0^{\lor})=r\neq 0,$$ where the last inequality follows from the assumption $p>\text{dim}\wedge^d(\g_1)\geq r$. It follows  that $(Y\omega)_0\neq 0$.
\end{proof}
Let $M=M_{\0}\oplus M_{\1}$ be a  restricted $\g$-module. If both $M_{\0}$ and $M_{\1}$ are rational $T$-modules such that the $\mathfrak h$-action is induced from the differential of the $T$-action. In addition, for all $t\in T$, $x\in \g_{\0}$ and $m\in M$, we have $t(xm)=\text{Ad}(t)(x)(tm)$. Then we call $M$ a $u(\g)-T$-module (\cite[1.8]{jj4}). For example, $u(\g)$ and $\bar{\mathfrak u}$ are both $u(\g)-T$-modules.\par

  Fix an order of the root vectors $y_{ij}$ of $\g_{-1}$ in the following lemma.
\begin{lemma} (1) Let $S=\Pi_{i<j}\text{ad} y_{ij} (\omega)\in u(\g)$ and let $\bar S$ be its image in $\bar {\mathfrak u}$. Then $\bar S\in \mathcal Z$ and $h(\bar S)=k\Theta$ for some $k\in F^*$.  \par
(2) For $z\in u(\g_{\0})^G$, let $S_z=\Pi_{i<j} \text{ad} y_{ij} (z\omega)\in u(\g)$ and let $\bar S_z$ be its image in $\bar{\mathfrak u}$. Then  $\bar S_z\in \mathcal Z$ and $h(\bar S_z)=h(\bar S)h(z)$.
\end{lemma}
\begin{proof} (1) Since $ F Y$ is a 1-dimensional $\text{ad}_{\g_{\0}}$-module,  $Y\otimes_F M$ is a  restricted simple $\g_{\0}$-module, for each  restricted simple $\g_{\0}$-module $M$. Then the mapping $Y\otimes_F -$ defines a bijection on the set of all restricted  simple $\g_{\0}$-modules.\par Now let $L(\nu)$ be a restricted simple $\g_{\0}$-module such that $Y\otimes_F L(\nu)$ is isomorphic to $L(\l)$ above. By the proof of Lemma 3.15, we have $(Y\omega)_0 Y\otimes m\neq 0$ for some $m\in L(\nu)$. Applying Proposition 2.3 to $K(L(\nu))$,  we have $$(Y\omega)_0Y\otimes m=Y\omega Y\otimes m=SY\otimes m=\bar SY\otimes m,$$ and hence $\bar S\neq 0$.\par
Since $\bar{\omega}\in \bar{\mathfrak u}_d$, $\g_1\bar{\omega}=\bar{\omega}\g_1=0$. We also have $g\bar{\omega}=\bar{\omega}$ for all $g\in G'$. Since $(Y\omega)_0$ is a $G$-invariant, it follows that $\bar{\omega}$ has $T$-weight $-\sigma$, where $\sigma $ is the $T$-weight of $Y$. Let $V\subseteq \bar{\mathfrak u}$ be an $\text{ad}_{\g}$-submodule generated by $\bar{\omega}$, which is also a $G$-submodule. Since $\bar S\neq 0$, we have $$V\cong u(\g)\otimes_{ u(\g_{\0}+\g_1)}F\bar{\omega}=K(F\bar{\omega}).$$ Clearly $V$ is a $u(\g)-T$-module.
Since  $V\cong \wedge (\g_{-1})\otimes F\bar\omega$, each element of $V$ may be written as $\sum_i u_i \bar\omega$ with $u_i\in \wedge ^i(\g_{-1})$.\par  Let $M\subseteq V$ be a simple $u(\g)-T$-submodule and let $\sum_i u_i \bar\omega\in M$ be a nonzero $T$-weight vector. By applying appropriate $y_{ij}$'s to this vector, we obtain $Y\bar\omega=\bar S\in M$. Since $M$ is simple, we obtain  $$M=\bar{\mathfrak u}\cdot\bar S=\wedge^{\leq d}(\g_1)\cdot\bar S.$$ Let $\wedge(\g_1)^+=:\sum_{i\geq 1} \wedge^i(\g_1)$. Then a short computation shows that $\wedge(\g_1)^+\bar S$ is a $u(\g)-T$-submodule of $M$. By comparing the $T$-weights we see that $\bar S\notin \wedge(\g_1)^+\bar S$, implying that  $M=F\bar S$, and hence $[\bar S, \g_{\bar 1}]=0$. From the proof of Lemma 3.16, we see that $Y$ is a $G'$-invariant. It follows that $\bar S$ is also a $G'$-invariant. This gives $\bar S\in\mathcal Z$,  since the $T$-weight of $\bar S$ is 0.\par
Note that $\bar S\in u^{2d}(\g)$. Then we have by Theorem 3.14 that $h(\bar S)=k\Theta +k'$ for some $k,k'\in F$.
Since $\bar S$ annihilates the trivial $\g$-module, in view of the fact that $\Theta (0)=0$,  we have $k'=h(\bar S)(0)=0$.
From above we have $\bar S K(L(\nu))\neq 0$, implying that  $h(\bar S)(\nu)=k\Theta (\nu)\neq 0$ and hence  $k\neq 0$.\par
 (2)  By a similar argument as above we have $\bar S_z\in \mathcal Z$, for $z\in u(\g_{\0})^G$. Let $\mu\in \Lambda_0$ and let $v_{\mu}\in K(L(\mu))$ be the maximal vector  of weight $\mu$. Since $FY$ is an $\text{ad}_{\g_{\0}}$-module, implying that $Yz=(z+c(z))Y$ for some $c(z)\in F$,  we have
 $$\begin{aligned} h(\bar S_z)(\mu) Y\otimes v_{\mu}&=S_zY\otimes v_{\mu}\\
 &=Yz\omega Y\otimes v_{\mu}\\&=(z+c(z))Y\omega Y\otimes v_{\mu}\\&
 =(z+c(z))\bar SY\otimes v_{\mu}\\&=\bar S(z+c(z))Y\otimes v_{\mu}\\&=\bar SY\otimes zv_{\mu}\\&=h(z)(\mu)\bar SY\otimes v_{\mu}\\&
 =h(z)(\mu)h(\bar S)(\mu)Y\otimes v_{\mu},\end{aligned}$$ implying that $h(\bar S_z)(\mu)=h(z)(\mu)h(\bar S)(\mu)$ for all $\mu\in\Lambda_0$, and hence  $h(\bar S_z)=h(z)h(\bar S)$.
\end{proof}
Immediately, we have the following conclusion.
\begin{corollary}  $$h(\mathcal Z)=F+\Theta u(\mathfrak h)^{W\cdot}.$$
\end{corollary}
\section{Representations of $\bar{\mathfrak u}$}
Fir each $\mu\in\Lambda_0$, we define an algebra homomorphism $\chi_{\mu}: \mathcal Z\longrightarrow F$ by $$\chi_{\mu}(z)=h(z)(\mu),\  z\in\mathcal Z.$$  Two weights $\mu, \nu\in \Lambda_0$ are referred to be  linked  and denoted  $\mu\sim \nu$, if there is $\sigma\in W$ such that $\mu+\rho=\sigma (\nu+\rho)$ (\cite{hu2}).\par
Let $U(\g'_{\0})$ be the universal enveloping algebra of $\g_{\0}'$ over $F$, and let $$U(\g'_{\0})^{G'}=\{u\in U(\g_{\0}')| g\cdot u=u\ \text{for all}\ g\in G'\}.$$
 Set $\mathfrak h_1=\mathfrak h\cap \g'_{\0}$.
 According to \cite[9.1-9.4]{jj}, there is an algebra isomorphism $$\pi: U(\g'_{\0})^{G'}\longrightarrow U(\mathfrak h_1)^{W\cdot}$$ with kernel $\mathfrak n_0^-U(\g_{\0}')+U(\g_{\0}')\mathfrak n_0^+$. \par

For each $\l\in\mathfrak h_1^*$, define an algebra homomorphism $\text{cen}_{\lambda}: U(\g_{\0}')^{G'}\longrightarrow F$ by $\text{cen}_{\l}(u)=\pi (u)(\lambda)$,  $u\in U(\g_{\0}')^{G'}$ (\cite[9.4]{jj}).\par
Let $\xi: U(\mathfrak h_1)\longrightarrow u(\mathfrak h_1)$ be the canonical homomorphism.
By \cite[Lemma 3.3]{hu2}, the restriction $\xi: U(\mathfrak h_1)^{W\cdot}\longrightarrow u(\mathfrak h_1)^{W\cdot}$ is an epimorphism.
It follows  that $\xi\pi:
U(\g'_{\0})^{G'}\longrightarrow u(\mathfrak h_1)^{W\cdot}$ is also an epimorphism.\par
Remark: The assumption on $p$ in \cite[Lemma 3.3]{hu2}  is satisfied in our case here, since $|W|=n!$.\par
For each $\l=\l_1\e_1+\cdots +\l_n\e_n\in\Lambda_0$. we have $$h_i^p(\l)-h_i(\l)=\l_i^p-\l_i=0\ \text{for}\  1\leq i\leq n.$$ Then we get $f(h)(\l)=\xi(f(h))(\l)$ for all $f(h)\in U(\mathfrak h)$.
 In particular, if $\l\in\Lambda_0\cap \mathfrak h_1^*$,  then we get $$\text{cen}_{\l}(u)=\pi(u)(\lambda)=\xi\pi (u)(\lambda)$$ for all $u\in U(\g_{\0}')^{G'}$.
\begin{theorem} Let $\mu, \nu\in \Lambda_0$. Then $\chi_{\mu}=\chi_{\nu}$ if and only if both $\mu$ and $\nu$ are atypical or both are typical with $\mu \sim \nu$.
\end{theorem}
\begin{proof} Suppose $\chi_{\mu}=\chi_{\nu}$. Then  we have $h(z)(\mu)=h(z)(\nu)$ for all $z\in\mathcal Z$. Write $$h(z)=\Theta f(h)+c, \  f(h)\in u(\mathfrak h)^{W\cdot},\ c\in F.$$ Then we have $\Theta (\mu)f(h)(\mu)=\Theta(\nu)f(h)(\nu)$ for all $f(h)\in  u(\mathfrak h)^{W\cdot}$.\par If $\mu$ is atypical, then we obtain $\Theta (\nu)f(h)(\nu)=0$ for all $f(h)\in u(\mathfrak h)^{W\cdot}$. Let $f(h)=1$. Then we get $\Theta(\nu)=0$, and hence $\nu$ is atypical.\par
 Suppose both $\mu$ and $\nu$ are typical.  Let $h(z)=\Theta$. Then we get  $\Theta(\mu)=\Theta(\nu)$ and hence  $f(h)(\mu)=f(h)(\nu)$ for all $f(h)\in u(\mathfrak h)^{W\cdot}$.  Since $u(\mathfrak h_1)^{W\cdot}\subseteq u(\mathfrak h)^{W\cdot}$, we obtain $f(h)(\mu)=f(h)(\nu)$ for all $f(h)\in u(\mathfrak h_1)^{W\cdot}$. It follows that
 $$\text{cen}_{\mu}(u)=\xi\pi(u)(\mu)=\xi\pi(u)(\nu)=\text{cen}_{\nu}(u)$$ for all $u\in U(\g'_{\0})^{G'}$, and hence
 $\text{cen}_{\mu}=\text{cen}_{\nu}$, and whence $\mu_{|\mathfrak h_1}=(w\cdot \nu)_{|\mathfrak h_1}$ for some $w\in W$,  by \cite[Corollary 9.4]{jj}.\par Since $\hbar=\sum_{i=1}^nh_i\in u(\mathfrak h)^{W\cdot}$,  we also have $\hbar (\mu)=\hbar (\nu)$, implying that $\mu(\hbar)=(w\cdot\nu)(\hbar)$, and hence $\mu=w\cdot \nu$.\par
If both $\mu$ and $\nu$ are atypical, then Corollary 3.26 gives $\chi_{\mu}=\chi_{\nu}$. Assume both $\mu$ and $\nu$ are typical and $\mu\sim \nu$. For each $z\in \mathcal Z$,  we have by Lemma 3.7 that $h(z)\in u(\mathfrak h)^{W\cdot}$, implying that $\chi_{\mu}(z)=\chi_{\nu}(z)$.

\end{proof}
For  $\mu\in \Lambda_0$, let $\bar{\mathfrak u}_{\mu}$ denote the superalgebra $\bar{\mathfrak u}/\bar{\mathfrak u}\text{ker}\chi_{\mu}$. Let $$\chi^0_{\mu}: u(\g_{\0})^G\longrightarrow F$$ be the homomorphism defined by $\chi^0_{\mu}(z)=h(z)(\mu), \ z\in u(\g_{\0})^G$.   Define the quotient algebra $u^0_{\mu}=: u(\g_{\0})/u(\g_{\0})\text{ker}\chi^0_{\mu}$.\par
 Note: It is easy to see that $\hbar$ is an element of $u(\g_{\0})^G$, but not necessarily  in $\mathcal Z$.\par
\begin{definition}If $M=M_{\0}\oplus M_{\1}$ is a $\bar{\mathfrak u}_{\mu}$-module and also a $u(\g)-T$-module,  then we call $M$ a $\bar{\mathfrak u}_{\mu}-T$-module.\end{definition}
Example. For $\mu\in\Lambda_0$, the baby Verma module $Z(\mu)$  is a $\bar{\mathfrak u}_{\mu}-T$-module.
 A $u^0_{\mu}-T$-module can be defined similarly.
\begin{theorem} Assume $\mu$ is typical. \par
(1) If $N$ is a simple $u^0_{\mu}-T$-module, then $K(N)$ is a simple $\bar{\mathfrak u}_{\mu}-T$-module.\par
(2) If $M$ is a simple $\bar{\mathfrak u}_{\mu}-T$-module, then $M^{\g_1}=\{m\in M|xm=0\ \text{for all}\  x\in\g_1\}$ is a simple $u^0_{\mu}-T$-module.
\end{theorem}
\begin{proof} (1) By definition $N$ is a restricted simple $\g_{\0}$-module annihilated by $\text{ker}\chi^0_{\mu}$. Let $v_{\l}\in N$ be the  maximal vector  of  weight $\l$. For each $z\in u(\g_{\0})^G$, we have $$zv_{\l}=h(z)v_{\l}=h(z)(\l)v_{\l}=\chi^0_{\l}(z)v_{\l}$$ and $$zv_{\l}=(z-\chi^0_{\mu}(z))v_{\l}+\chi^0_{\mu}(z)v_{\l}=\chi^0_{\mu}(z)v_{\l},$$ since $z-\chi^0_{\mu}(z)\in\text{ker}\chi^0_{\mu}$. Therefore, we get $\chi^0_{\l}=\chi^0_{\mu}$. Applying \cite[Theorem 3.1]{hu2}, we get $\l_{|\mathfrak h_1}\sim \mu_{|\mathfrak h_1}$,  and hence $\l_{|\mathfrak h_1}=(\sigma\cdot \mu)_{|\mathfrak h_1}$ for some $\sigma\in W$.\par
By \cite[Lemma 3.2, 3.3]{hu2}, the Harish-Chandra morphism $h: u(\g_{\0}')^{G'}\longrightarrow u(\mathfrak h_1)^{W\cdot}$ is an epimorphism. Note that $$u(\g_{\0})^G\cong u(\g_{\0}')^{G'}\otimes u(F\hbar)$$ and $h(\hbar)=\hbar\in u(\mathfrak h)^{W\cdot}$. Then $h: u(\g_{\0})^{G}\longrightarrow u(\mathfrak h)^{W\cdot}$ is also an epimorphism. From $\chi^0_{\l}(\hbar)=\chi^0_{\mu}(\hbar)$, we obtain $\hbar (\l)=\hbar (\mu)$, which gives $\l(\hbar)=(\sigma\cdot \mu)(\hbar)$, since $\hbar \in u(\mathfrak h)^{W\cdot}$.  Thus, we obtain $\l\sim \sigma\cdot \mu$. Then
 $\l$ is also typical by Theorem 4.1. Hence $K(N)$ is simple as $u(\g)$-module by Proposition 2.3, and
  also as $\bar{\mathfrak u}$-module by Lemma 3.2.\par
  Let the $T$-action on $K(N)$ be induced from that of $N$. It remains to show that $K(N)$ is annihilated by $\text{ker}\chi_{\mu}$. To see this, let $z\in \text{ker}\chi_{\mu}$. Then we have $$\begin{aligned} zv_{\l}&=h(z)v_{\l}\\&=h(z)(\l)v_{\l}\\&=h(z)(\mu)v_{\l}\\&=\chi_{\mu}(z)v_{\l}\\&=0,\end{aligned}$$ and hence $zK(N)=0$.\par
(2) Let $M=M_{\0}\oplus M_{\1}$ and let $m\in M_{\0}\cup M_{\1}$ be a nonzero $T$-weight vector. Set $M'=\wedge (\g_1)m\subseteq M$. Since  $\wedge^i(\g_1)=0$ in $\bar{\mathfrak u}$ is for $i>d$, we have $$M'=\wedge^{\leq d} (\g_1)m.$$  Let $i_0$ be the largest integer such that $\wedge^{i_0}(\g_1)m\neq 0$. Then we have  $0\leq i_0\leq d$ and $\g_1 \wedge^{i_0}(\g_1)m=0$. Therefore we have $M^{\g_1}\neq 0$.  Clearly $M^{\g_1}$ is $\mathbb Z_2$-graded.\par
 It is easy to see that $M^{\g_1}$ is a $T$-submodule and a $\g_{\0}$-submodule of $M$. Since $M$ is  a restricted $\g_{\0}$-module, so is $M^{\g_1}$. It is no loss of generality to assume that $M^{\g_1}_{\0}\neq 0$.\par  Let $L(\l)\subseteq M^{\g_1}_{\0}$ be a simple $u(\g_{\0})$-submodule of highest weight $\l$. Then $L(\l)$ is a $u(\g_{\0})-T$-submodule. The inclusion $L(\l)\subseteq M_{\0}$ induces a $u(\g)$-module homomorphism from $K(L(\l))$ onto $M$. Since $M$ is a $\bar{\mathfrak u}_{\mu}$-module,  by applying $z\in \mathcal Z$ to the maximal vector $v_{\l}\in L(\l)$ we obtain $\chi_{\l}=\chi_{\mu}$, implying that $\l\sim \mu$ by Theorem 4.1. It follows that $\l$ is also typical, and hence $K(L(\l))$ is simple, and whence $M\cong K(L(\l))$. We simply  write $M=K(L(\l))$. Then   $L(\l)\subseteq M^{\g_1}$.\par  We claim that $M^{\g_1}= L(\l)$. To see this, note that $$M=L(\l)\oplus \g_{-1}M.$$ Let $N=M^{\g_1}\cap \g_{-1}M$. Then we have $$M^{\g_1}=L(\l)+N.$$  Let $\tilde N\subseteq M$ be the $\g$-submodule generated by $N$. By the definition of $K(L(\l))$ we  have $\tilde N\subseteq \g_{-1}M$, and hence $\tilde N\cap L(\l)=0$. Since $M$ is simple,  we have $\tilde N=0$ and hence $N=0$. So the claim follows.\par
 It remains to show that $L(\l)$ is annihilated by $\text{ker}\chi^0_{\mu}$.  Let $z\in u(\g_{\0})^G$.  Since $\l\sim \mu$,  we have $$\begin{aligned} zv_{\l}&=h(z)(\l)v_{\l}\\ &=h(z)(\mu)v_{\l}\\ &=\chi^0_{\mu}(z)v_{\l}.\end{aligned}$$  Then $z$ acts on $L(\l)$ as a multiplication by $\chi^0_{\mu}(z)$ or, equivalently, $L(\l)$ is annihilated by $\text{ker}\chi^0_{\mu}$. It follows that $L(\l)$ is a simple $u^0_{\mu}-T$-module.
\end{proof}
\vskip 0.4truein
\def\refname{

\end{document}